\documentclass[a4paper,9pt,oneside]{article}

\usepackage[utf8]{inputenc}
\usepackage[T1]{fontenc}
\usepackage[english]{babel}
\usepackage[a4paper,top=2.8cm,bottom=3.2cm,left=3.9cm,right=3.3cm]{geometry}
\usepackage[utf8]{inputenc}
\usepackage{graphicx}
\usepackage{amsmath}
\usepackage{amssymb}
\usepackage{amsthm}
\usepackage{color}
\usepackage{version}
\usepackage{mathtools}
\usepackage{tikz}
\usepackage{amsfonts}
\usepackage{bm}
\usepackage{enumitem}
\usepackage{calligra}
\usepackage{calrsfs}
\usepackage{mathabx}
\usepackage{tikz}
\usepackage{caption}
\usepackage{mathrsfs}
\usepackage{cancel}
\usetikzlibrary{arrows}
\theoremstyle{plain}
\newtheorem{thm}{Theorem}
\newtheorem{lem}[thm]{Lemma}
\newtheorem{prop}[thm]{Proposition}
\newtheorem{cor}[thm]{Corollary}
\theoremstyle{definition}

\newtheorem{defn}[thm]{Definition}

\theoremstyle{remark}

\newcommand{\de}{\mathrm{d}}
\newcommand{\Le}{L}

\newcommand{\qdot}{\Delta^\frac{1}{2}}

\DeclareMathAlphabet{\mathcal}{OMS}{zplm}{m}{n}

\title{Unitarization of the Horocyclic Radon Transform on Homogeneous Trees}
\date{}
\author{Francesca Bartolucci\thanks{Department of Mathematics, ETH Zurich, Raemistrasse 101, 8092 Zurich, Switzerland, e-mail: francesca.bartolucci@sam.math.ethz.ch} \and Filippo De Mari\thanks{Department of Mathematics \& MaLGa Center,
   University of Genoa, Via Dodecaneso 35, 16146 Genova, Italy, email: demari@dima.unige.it, m.monti@dima.unige.it} \and Matteo Monti\footnotemark[2]}
\begin{document}

\maketitle

		\begin{abstract}
		Following previous work in the continuous setup, we construct the unitarization of the horocyclic Radon transform on a homogeneous tree $X$ and we show that it intertwines the quasi regular representations of the group of isometries of $X$ on the tree itself and on the space of horocycles.
	\end{abstract}

\vspace{0.5mm}

\vspace{1mm}	

	\section*{Introduction}

The horocyclic Radon transform on homogeneous trees was introducted by P.~Cartier~\cite{cart} and   studied by A.~Fig\`a-Talamanca and M.A.~Picardello \cite{ftp}, W.~Betori, J.~Faraut and M.~Pagliacci \cite{bfp}, M.~Cowling, S.~Meda and A.G.~Setti \cite{cms},  E.~Casadio Tarabusi, J.~Cohen and F.~Colonna \cite{ctcc}, and A.~Veca \cite{veca}, to name a few. Some of the typical issues considered are inversion formul{\ae}  and range problems. In this paper, we  treat the unitarization problem, that is the determination of some kind of  pseudo-differential operator such that the pre-composition with the Radon transform yields a unitary operator. This is a classical aspect of Radon theory,  addressed first by 
Helgason \cite{radon}  in the case of the polar Radon transform.  In  \cite{abdd}, the  authors consider a 
general setup that may be recast as a variation of the setup of dual pairs $(X,\Xi)$ {\it \`a la} Helgason~\cite{hel1970} or~\cite{radon}.  They prove a general result concerning the unitarization of the Radon transform ${\mathcal R}$ from $L^2(X,\de x)$ to  $L^2(\Xi,\de\xi)$ and then show  that the resulting unitary operator intertwines the quasi regular representations of $G$ on $L^2(X,\de x)$ and $L^2(\Xi,\de\xi)$.  As already mentioned, this unitarization really means first composing (the closure of) ${\mathcal R}$ with a suitable pseudo-differential operator and then extending this composition to a unitary map, as it is done in the existing and well known precedecessors of this result \cite{radon}, \cite{solmon76}. The construction of the unitary map is the crucial step in finding the explicit and new inversion formula for the Radon transform which holds under the hypotheses assumed in~\cite{abdd}, primarily the facts that the quasi regular representations of $G$ on $X$ and $\Xi$  are both irreducible and both square integrable. This kind of consequence was one of our main motivations for adressing the issue of building the unitary ``extension'' in the context of homogeneous trees. The techniques used in~\cite{abdd}
cannot be transferred directly to the case of homogeneous trees primarily because the quasi regular representation is not irreducible, much less square integrable, a fact that we explicitly recall in Appendix~\ref{appendiceirri}. Hence, we adopt here a combination of the 
classical approach followed by Helgason  in the symmetric space case~\cite{gass} and the  techniques that have been exploited in \cite{acha}.
The paper is organized in three sections. In Section 1, we present the main notions and the relevant results in the theory of homogeneous trees.  Then, we give a brief overview 
of the Helgason-Fourier transform. In Section 2, we  recall the horocyclic Radon transform on homogeneous trees, we present its link with the Helgason-Fourier transform and we show its intertwining properties with quasi regular representations. Finally,
in Section 3, we prove the unitarization theorem for the horocyclic Radon transform. 
For the reader's convenience, we add two short appendices. In the first, we briefly recall the notion of dual pair and in the second we indicate why the quasi regular representation on $L^2(X)$ is not irreducible.
\\

\section{Preliminaries} 

In Subsections~\ref{HT} through \ref{R} we recall the basic definitions and facts that will be used throughout,  focusing on the space of horocycles. In particular, we describe the various group actions that are relevant in order to apply the machinery of dual pairs that was devised by Helgason (see Appendix~\ref{appendiceHel}).  Subsection~\ref{FT} is devoted to a brief overview 
	of the Helgason-Fourier transform. Standard references for these are \cite{bfp}, \cite{cms} and \cite{ftn}.


	\subsection{Homogeneous trees}\label{HT}

A \textit{graph} is a pair $(X,\mathfrak{E})$, where $X$ is the set of \textit{vertices} and $\mathfrak{E}$ is the family of \textit{edges}, where an edge is a two-element subset of $X$. We often think of an edge as a segment joining two vertices. If two vertices are joint by a segment, they are called \textit{adjacent}. A \textit{tree} is an undirected, connected, loop-free graph. In this paper we are interested in homogeneous trees. A $q$-\textit{homogeneous} tree is a tree in which each vertex has exactly $q+1$ adjacent vertices. If $q\geq 1$, a $q$-homogeneous tree is infinite. From now on, we suppose $q\geq 2$ in order to exclude trivial cases, that is, segments and lines.
\par
Given $u,v\in X$ with $u\neq v$, we denote by $[u,v]$ the unique ordered $t$-uple $(x_0=u,x_1,\dots, x_{t-1}=v)\in X^t$, where $\{x_i,x_{i+1}\}\in\mathfrak{E}$ and all the $x_i$ are distinct. We call $[u,v]$ a (finite) $t$\textit{-chain} and we think of it as a path starting at $u$ and ending at $v$ or, equivalently, as the finite sequence of consecutive $2$-chains $[u,x_1],[x_1,x_2],\dots,[x_{t-2},v]$. With slight abuse of notation, if $[u,v]=(x_0,\dots,x_{t-1})$ we write $u,v,x_i\in [u,v]$ and $[u,x_i]\cup[x_i,v]$, $i\in\{1,\dots,t-2\}$. In particular, if $u$ and $v$ are adjacent, both $[u,v],[v,u]\in X^2$ are oriented, unlike the edge $\{u,v\}\in \mathfrak{E}$ which is not. A homogeneous tree $X$ carries a natural distance $d\colon X\times X\to\mathbb{N}$, where for every $u,v\in X$ the distance $d(u,v)$ is the number of  $2$-chains in the path $[u,v]$.
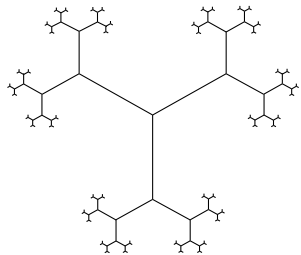
\begin{figure}[h]
	\begin{center}
\begin{tikzpicture}[x=0.75pt,y=0.75pt,yscale=-1,xscale=1, scale=0.7]
\draw    (157.65,60.55) -- (210.04,90.22) ;
\draw    (210.04,90.22) -- (210.04,150.87) ;
\draw    (210.04,90.22) -- (262.45,60.55) ;
\draw    (262.45,60.55) -- (289.25,75.35) ;
\draw    (157.65,60.55) -- (131.25,75.35) ;
\draw    (157.65,60.55) -- (157.65,30.15) ;
\draw    (262.45,60.55) -- (262.45,30.1) ;
\draw    (210.04,150.87) -- (236.84,165.67) ;
\draw    (210.04,150.87) -- (183.64,165.67) ;
\draw    (289.45,75.16) -- (289.45,90.25) ;
\draw    (289.45,75.16) -- (302.49,67.78) ; 
\draw    (302.49,67.78) -- (309.16,71.46) ;
\draw    (302.49,67.78) -- (302.49,60.2) ; 
\draw    (289.45,90.25) -- (296.12,93.93) ;
\draw    (289.45,90.25) -- (282.88,93.93) ;
\draw    (249.41,23.18) -- (262.45,30.56) ;
\draw    (262.45,30.56) -- (275.49,23.18) ;
\draw    (275.49,23.18) -- (282.16,26.86) ;
\draw    (249.41,23.18) -- (242.85,26.86) ;
\draw    (249.41,23.18) -- (249.41,15.62) ;
\draw    (275.49,23.18) -- (275.49,15.6) ;
\draw    (144.66,23.36) -- (157.7,30.1) ; 
\draw    (157.7,30.1) -- (170.74,23.36) ;
\draw    (170.74,23.36) -- (177.41,26.72) ;
\draw    (144.66,23.36) -- (138.1,26.72) ;
\draw    (144.66,23.36) -- (144.66,16.46) ;
\draw    (170.74,23.36) -- (170.74,16.45) ;
\draw    (118.1,67.89) -- (131.14,75.27) ;
\draw    (131.14,75.27) -- (131.14,90.36) ;
\draw    (118.1,67.89) -- (111.54,71.57) ;
\draw    (118.1,67.89) -- (118.1,60.33) ;
\draw    (131.14,90.36) -- (137.81,94.04) ;
\draw    (131.14,90.36) -- (124.57,94.04) ;
\draw    (236.81,165.54) -- (236.81,180.63) ;
\draw    (236.81,165.54) -- (249.84,158.15) ;
\draw    (249.84,158.15) -- (256.51,161.84) ;
\draw    (249.84,158.15) -- (249.84,150.58) ;
\draw    (236.81,180.63) -- (243.47,184.31) ;
\draw    (236.81,180.63) -- (230.24,184.31) ;
\draw    (170.91,158.28) -- (183.95,165.66) ;
\draw    (183.95,165.66) -- (183.95,180.75) ;
\draw    (170.91,158.28) -- (164.35,161.96) ;
\draw    (170.91,158.28) -- (170.91,150.72) ;
\draw    (183.95,180.75) -- (190.62,184.43) ;
\draw    (183.95,180.75) -- (177.38,184.43) ; 
\draw    (296.19,93.87) -- (296.19,97.71) ;
\draw    (296.19,93.87) -- (299.5,92) ;
\draw    (299.5,92) -- (301.19,92.93) ;
\draw    (299.5,92) -- (299.5,90.07) ;
\draw    (296.24,97.71) -- (297.93,98.64) ;
\draw    (296.24,97.71) -- (294.57,98.64) ;
\draw    (309.14,71.45) -- (309.14,75.29) ;
\draw    (309.14,71.45) -- (312.46,69.57) ;
\draw    (312.46,69.57) -- (314.15,70.51) ;
\draw    (312.46,69.57) -- (312.46,67.65) ;
\draw    (309.19,75.29) -- (310.89,76.22) ;
\draw    (309.19,75.29) -- (307.52,76.22) ;
\draw    (282.05,26.7) -- (282.05,30.53) ;
\draw    (282.05,26.7) -- (285.37,24.82) ;
\draw    (285.37,24.82) -- (287.06,25.75) ;
\draw    (285.37,24.82) -- (285.37,22.89) ;
\draw    (282.1,30.53) -- (283.8,31.47) ;
\draw    (282.1,30.53) -- (280.43,31.47) ;
\draw    (176.67,26.47) -- (176.67,30.31) ;
\draw    (176.67,26.47) -- (179.99,24.6) ;
\draw    (179.99,24.6) -- (181.68,25.53) ;
\draw    (179.99,24.6) -- (179.99,22.67) ;
\draw    (176.73,30.31) -- (178.42,31.24) ;
\draw    (176.73,30.31) -- (175.06,31.24) ;
\draw    (137.87,94.07) -- (137.87,97.91) ;
\draw    (137.87,94.07) -- (141.19,92.2) ;
\draw    (141.19,92.2) -- (142.88,93.13) ;
\draw    (141.19,92.2) -- (141.19,90.27) ;
\draw    (137.93,97.91) -- (139.62,98.84) ;
\draw    (137.93,97.91) -- (136.26,98.84) ;
\draw    (243.47,184.27) -- (243.47,188.11) ;
\draw    (243.47,184.27) -- (246.79,182.4) ;
\draw    (246.79,182.4) -- (248.48,183.33) ;
\draw    (246.79,182.4) -- (246.79,180.47) ;
\draw    (243.53,188.11) -- (245.22,189.04) ;
\draw    (243.53,188.11) -- (241.86,189.04) ;
\draw    (256.27,161.67) -- (256.27,165.51) ;
\draw    (256.27,161.67) -- (259.59,159.8) ;
\draw    (259.59,159.8) -- (261.28,160.73) ;
\draw    (259.59,159.8) -- (259.59,157.87) ;
\draw    (256.33,165.51) -- (258.02,166.44) ;
\draw    (256.33,165.51) -- (254.66,166.44) ;
\draw    (190.47,184.27) -- (190.47,188.11) ;
\draw    (190.47,184.27) -- (193.79,182.4) ;
\draw    (193.79,182.4) -- (195.48,183.33) ;
\draw    (193.79,182.4) -- (193.79,180.47) ;
\draw    (190.53,188.11) -- (192.22,189.04) ;
\draw    (190.53,188.11) -- (188.86,189.04) ;
\draw    (279.63,91.93) -- (282.94,93.81) ;
\draw    (282.94,93.81) -- (282.94,97.64) ;
\draw    (279.63,91.93) -- (277.96,92.87) ;
\draw    (279.63,91.93) -- (279.63,90.01) ;
\draw    (282.99,97.64) -- (284.69,98.58) ;
\draw    (282.99,97.64) -- (281.32,98.58) ; 
\draw    (239.43,25.13) -- (242.74,27.01) ;
\draw    (242.74,27.01) -- (242.74,30.84) ;
\draw    (239.43,25.13) -- (237.76,26.07) ;
\draw    (239.43,25.13) -- (239.43,23.21) ;
\draw    (242.79,30.84) -- (244.49,31.78) ;
\draw    (242.79,30.84) -- (241.12,31.78) ;
\draw    (134.83,24.93) -- (138.14,26.81) ;
\draw    (138.14,26.81) -- (138.14,30.64) ;	
\draw    (134.83,24.93) -- (133.16,25.87) ;
\draw    (134.83,24.93) -- (134.83,23.01) ;
\draw    (138.19,30.64) -- (139.89,31.58) ;
\draw    (138.19,30.64) -- (136.52,31.58) ;
\draw    (108.23,69.53) -- (111.54,71.41) ;
\draw    (111.54,71.41) -- (111.54,75.24) ;
\draw    (108.23,69.53) -- (106.56,70.47) ;
\draw    (108.23,69.53) -- (108.23,67.61) ;
\draw    (111.59,75.24) -- (113.29,76.18) ;
\draw    (111.59,75.24) -- (109.92,76.18) ;
\draw    (121.23,91.93) -- (124.54,93.81) ;
\draw    (124.54,93.81) -- (124.54,97.64) ;
\draw    (121.23,91.93) -- (119.56,92.87) ;
\draw    (121.23,91.93) -- (121.23,90.01) ;
\draw    (124.59,97.64) -- (126.29,98.58) ;
\draw    (124.59,97.64) -- (122.92,98.58) ;
\draw    (174.23,182.53) -- (177.54,184.41) ;
\draw    (177.54,184.41) -- (177.54,188.24) ;
\draw    (174.23,182.53) -- (172.56,183.47) ;
\draw    (174.23,182.53) -- (174.23,180.61) ;
\draw    (177.59,188.24) -- (179.29,189.18) ;
\draw    (177.59,188.24) -- (175.92,189.18) ;
\draw    (161.23,160.13) -- (164.54,162.01) ;
\draw    (164.54,162.01) -- (164.54,165.84) ;
\draw    (161.23,160.13) -- (159.56,161.07) ;
\draw    (161.23,160.13) -- (161.23,158.21) ;
\draw    (164.59,165.84) -- (166.29,166.78) ;
\draw    (164.59,165.84) -- (162.92,166.78) ;
\draw    (226.83,182.33) -- (230.14,184.21) ;
\draw    (230.14,184.21) -- (230.14,188.04) ;
\draw    (226.83,182.33) -- (225.16,183.27) ;
\draw    (226.83,182.33) -- (226.83,180.41) ;
\draw    (230.19,188.04) -- (231.89,188.98) ;
\draw    (230.19,188.04) -- (228.52,188.98) ;
\draw    (299.23,58.93) -- (302.54,60.81) ;
\draw    (302.54,60.81) -- (305.86,58.93) ;
\draw    (305.86,58.93) -- (307.55,59.87) ;
\draw    (299.23,58.93) -- (297.56,59.87) ;
\draw    (299.23,58.93) -- (299.23,57.01) ;
\draw    (305.86,58.93) -- (305.86,57) ;
\draw    (114.83,59.23) -- (118.14,61.11) ;
\draw    (118.14,61.11) -- (121.46,59.23) ;
\draw    (121.46,59.23) -- (123.15,60.17) ;
\draw    (114.83,59.23) -- (113.16,60.17) ;
\draw    (114.83,59.23) -- (114.83,57.31) ;
\draw    (121.46,59.23) -- (121.46,57.3) ;
\draw    (246.03,14.33) -- (249.34,16.21) ;
\draw    (249.34,16.21) -- (252.66,14.33) ;
\draw    (252.66,14.33) -- (254.35,15.27) ;
\draw    (246.03,14.33) -- (244.36,15.27) ;
\draw    (246.03,14.33) -- (246.03,12.41) ;
\draw    (252.66,14.33) -- (252.66,12.4) ;
\draw    (272.23,13.73) -- (275.54,15.61) ;
\draw    (275.54,15.61) -- (278.86,13.73) ;
\draw    (278.86,13.73) -- (280.55,14.67) ;
\draw    (272.23,13.73) -- (270.56,14.67) ;
\draw    (272.23,13.73) -- (272.23,11.81) ;
\draw    (278.86,13.73) -- (278.86,11.8) ;
\draw    (167.43,14.53) -- (170.74,16.41) ;
\draw    (170.74,16.41) -- (174.06,14.53) ;
\draw    (174.06,14.53) -- (175.75,15.47) ;
\draw    (167.43,14.53) -- (165.76,15.47) ;
\draw    (167.43,14.53) -- (167.43,12.61) ;
\draw    (174.06,14.53) -- (174.06,12.6) ;
\draw    (141.43,14.93) -- (144.74,16.81) ;
\draw    (144.74,16.81) -- (148.06,14.93) ;
\draw    (148.06,14.93) -- (149.75,15.87) ;
\draw    (141.43,14.93) -- (139.76,15.87) ;
\draw    (141.43,14.93) -- (141.43,13.01) ;
\draw    (148.06,14.93) -- (148.06,13) ;
\draw    (167.63,149.03) -- (170.94,150.91) ;
\draw    (170.94,150.91) -- (174.26,149.03) ;
\draw    (174.26,149.03) -- (175.95,149.97) ;
\draw    (167.63,149.03) -- (165.96,149.97) ;
\draw    (167.63,149.03) -- (167.63,147.11) ;
\draw    (174.26,149.03) -- (174.26,147.1) ;
\draw    (246.63,148.63) -- (249.94,150.51) ;
\draw    (249.94,150.51) -- (253.26,148.63) ;
\draw    (253.26,148.63) -- (254.95,149.57) ;
\draw    (246.63,148.63) -- (244.96,149.57) ;
\draw    (246.63,148.63) -- (246.63,146.71) ;
\draw    (253.26,148.63) -- (253.26,146.7) ;
\end{tikzpicture}
	\end{center}
	\caption{A portion of a $2$-homogeneous tree}
\end{figure}

\par

	\subsection{The boundary of a homogeneous tree}\label{BDRY}

An \textit{infinite chain} is an infinite sequence $(x_i)_{i\in\mathbb{N}}$ of vertices of $X$ such that, for every $i\in\mathbb{N}$, $\:d(x_i,x_{i+1})=1$ and $x_i\neq x_{i+2}$. We denote by $c(X)$ the set of infinite chains on $X$. We say that two chains $(x_i)_{i\in\mathbb{N}}$ and $(y_i)_{i\in\mathbb{N}}$ are equivalent if there exist $m\in\mathbb{Z}$ and $N\in\mathbb{N}$ such that $x_i=y_{i+m}$ for every $i\geq N$ and, in such case, we write $(x_i)_{i\in\mathbb{N}}\sim (y_i)_{i\in\mathbb{N}}$. The \textit{boundary} of $X$ is the space $\Omega$ of equivalence classes $c(X)/\sim$. Observe that an infinite chain identifies uniquely a point of the boundary, which may be thought of as a point at infinity. In fact, it is well known \cite{cohen} that a homogeneous tree of even order $q+1$ can be isometrically embedded in the unit disc, the latter endowed with its hyperbolic metric, in such a way that the limit points of infinite chains correspond  precisely to the points of the unit circle, the topological boundary of the unit disc.

	\par
	We denote by $p$ the canonical projection of $c(X)$ onto $\Omega$. For $v\in X$ and $\omega\in\Omega$ we write $[v,\omega)$ for the unique chain $(x_i)_{i\in\mathbb{N}}$ starting at $v$, i.e.  $x_0=v$,  and ``pointing at'' the boundary point $\omega$, i.e. $p((x_i)_{i\in\mathbb{N}})=\omega$. Furthermore, given $\omega_1,\omega_2\in \Omega$ with $\omega_1\ne\omega_2$, we denote by $(\omega_1,\omega_2)$ the unique infinite sequence of vertices $(x_i)_{i\in\mathbb{Z}}$ such that $(x_{-i})_{i\in\mathbb{N}}\in\omega_1$ and $(x_i)_{i\in\mathbb{N}}\in\omega_2$, with $x_1\neq x_{-1}$, and we call it a doubly infinite chain. 
	We fix an arbitrary reference point $o\in X$. The boundary $\Omega$ is endowed with the topology (independent of the reference point) generated by the open sets
	\[\Omega(u)=\{\omega\in\Omega:u\in[o,\omega)\},\qquad u\in X.\]
	With this topology, $\Omega$ is a compact topological space.
	For later use, we remark that in every class $\omega\in\Omega$, there is a unique infinite chain $[o,\omega)$  starting at $o$, and we denote by
	\[
	\Gamma_o=\{[o,\omega):\omega\in\Omega\}.
	\]
	the set of all infinite chains starting at $o$. Clearly $\Gamma_o$ and $\Omega$ may be identified.

		\subsection{Horocycles}\label{HORO}

	A $2$-chain $[v,u]$ is said to be positively oriented with respect to  $\omega\in\Omega$ if $u\in[v,\omega)$, otherwise we say that $[v,u]$ is negatively oriented.

For $\omega\in\Omega$ and $v,u\in X$, we denote by $\kappa_\omega(v,u)\in\mathbb{Z}$  the so-called \textit{horocyclic index} of $v$ and $u$ w.r.t. $\omega$, namely  the number of positively oriented 2-chains (w.r.t. $\omega$) in $[v,u]$ minus the number of negatively oriented 2-chains (w.r.t. $\omega$) in $[v,u]$. Clearly, $|\kappa_{\omega}(v,u)|\leq d(v,u)$. It is easy to verify that,
for every $v,u,x\in X$ and for every $\omega\in\Omega$, 
\begin{equation}\label{eq:kprop}
	\kappa_\omega(v,x)=\kappa_\omega(v,u)+\kappa_\omega(u,x).
\end{equation}
Furthermore, we have the following result. 
\begin{prop}[\cite{cms}]\label{prop:limitk}
	Let $v\in X$ and $\omega\in\Omega$. If $[v,\omega)=(x_i)_{i\in\mathbb{N}}$, then for every $x\in X$
	\[\kappa_\omega(v,x)=\lim_{i\rightarrow\infty}(i-d(x,x_i)).\]
\end{prop}      
\begin{proof}
	We fix $x\in X$ and we observe that
	\[\lim_{i\rightarrow\infty}(i-d(x,x_i))=\lim_{i\rightarrow\infty}(d(v,x_i)-d(x,x_i)).\]
	Since $[v,\omega)\sim [x,\omega)$, then there exists $N\in\mathbb{N}$ such that
	$x_N\in[x,\omega)$ and $x_{N-1}\not\in[x,\omega)$,
	with the understanding that if $v\in[x,\omega)$ then $N=0$. Thus, for all $i\geq N$
	\[d(v,x_i)-d(x,x_i)=d(v,x_N)-d(x,x_N)\]
	and then 
	\[\lim_{i\rightarrow\infty}(d(v,x_i)-d(x,x_i))=d(v,x_N)-d(x,x_N).\]
	Furthermore,  $[v,x]=[v,x_N]\cup[x_N,x]$, where $[v,x_N]$ is the union of positively oriented 2-chains and $[x_N,x]$ is the union of negatively oriented 2-chains. Hence, 
	\begin{align*}
		\kappa_\omega(v,x)&=d(v,x_N)-d(x,x_N)=\lim_{i\rightarrow\infty}(d(v,x_i)-d(x,x_i))=\lim_{i\rightarrow\infty}(i-d(x,x_i))
	\end{align*}
	and this concludes the proof. 
\end{proof}

We are now in a position to introduce the  horocycles. 
\begin{defn}
	For $\omega\in\Omega$, $v\in X$ and $n\in\mathbb{Z}$, the \textit{horocycle} tangent to $\omega$ of index $n$ with respect to the vertex $v$ is the subset of $X$ defined as 
	\[
	h_{\omega,n}^v=\{x\in X:\kappa_\omega(v,x)=n\}.
	\]
	We denote by $\Xi$ the set of horocycles.
\end{defn} 
\begin{figure}[h]
	\begin{center}
		\begin{tikzpicture}[scale=0.6]
			\draw (-6,0)--(-5.5,1);
			\draw (-5,0)--(-5.5,1);
			\draw (-4.5,0)--(-4,1);
			\draw (-3.5,0)--(-4,1);
			\draw (-3,0)--(-2.5,1);
			\draw (-2,0)--(-2.5,1);
			\draw (-1.5,0)--(-1,1);
			\draw (-0.5,0)--(-1,1);
			\draw (6,0)--(5.5,1);
			\draw (5,0)--(5.5,1);
			\draw (4.5,0)--(4,1);
			\draw (3.5,0)--(4,1);
			\draw (3,0)--(2.5,1);
			\draw (2,0)--(2.5,1);
			\draw (1.5,0)--(1,1);
			\draw (0.5,0)--(1,1);
			\draw (-5.5,1)--(-4.75,2);
			\draw (-4,1)--(-4.75,2);
			\draw (-2.5,1)--(-1.75,2);
			\draw (-1,1)--(-1.75,2);
			\draw (1,1)--(1.75,2);
			\draw (2.5,1)--(1.75,2);
			\draw (4,1)--(4.75,2);
			\draw (5.5,1)--(4.75,2);
			\draw (1.75,2)--(3.25,3);
			\draw (4.75,2)--(3.25,3);
			\draw (-1.75,2)--(-3.25,3);
			\draw (-4.75,2)--(-3.25,3);
			\draw (-3.25,3)--(0,4);
			\draw (3.25,3)--(0,4);
			\draw[dashed] (0,4)--(0,5);
			\draw[dashed,black] (-6.5,4)--(6.5,4);
			\draw[dashed,black] (-6.5,3)--(6.5,3);
			\draw[dashed,black] (-6.5,2)--(6.5,2);
			\draw[dashed,black] (-6.5,1)--(6.5,1);
			\draw[dashed,black] (-6.5,0)--(6.5,0);
			\node[anchor=west,black] at (6.5,3) { $h_{\omega,0}^v$};
			\node[anchor=west,black] at (6.5,2) { $h_{\omega,-1}^v$};
			\node[anchor=west,black] at (6.5,1) { $h_{\omega,-2}^v$};
			\node[anchor=west,black] at (6.5,0) { $h_{\omega,-3}^v$};
			\node[anchor=west,black] at (6.5,4) { $h_{\omega,1}^v$};
			\filldraw[black] (-3.25,3) circle (2pt) node[anchor=south east] {$v$};
			\draw[->,black] (0.2,4.5)--(0.2,5);
			\node[anchor=west,black] at (0.2,4.75) { $\omega$};
		\end{tikzpicture}
	\end{center}
	\caption{A part of a 2-homogeneous tree containing portions of horocycles (unions of vertices lying on dashed lines) which are tangent to $\omega$.}
\end{figure}
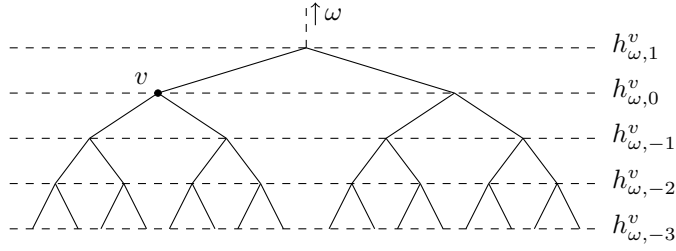
It follows immediately from \eqref{eq:kprop} that for every $v,u\in X$, $n\in\mathbb{Z}$ and 
$\omega\in\Omega$
\begin{equation}\label{eq:hprop}
	h_{\omega,n}^v=h_{\omega,n+\kappa_\omega(u,v)}^u.
\end{equation}
Hence the mapping $(v,\omega,n)\mapsto h_{\omega,n}^v$ is not injective and so $\Xi$ is not well parametrized by $X\times\Omega\times\mathbb{Z}$. However, for fixed $v\in X$, the map $(\omega,n)\mapsto h_{\omega,n}^v$ is actually bijective, so that $\Xi$ may be identified with $\Omega\times\mathbb{Z}$. Formally, for every $v\in X$, there is a bijection 
\begin{equation}\label{Psi}
	\Psi_v\colon\Omega\times\mathbb{Z}\to \Xi,\qquad 
\Psi_v(\omega,n)= h_{\omega,n}^v
\end{equation}
 and, for every fixed $\omega\in\Omega$, $X$ can be covered disjointly as
\begin{equation}\label{eq:disjointunion}
X=\bigcup_{n\in\mathbb{Z}}h^v_{\omega,n}.
\end{equation}

By equality \eqref{eq:hprop}, for each pair of vertices $u,\, v\in X$ 
\begin{equation*}
	\Psi_u^{-1}\circ\Psi_v(\omega,n)=(\omega,n+\kappa_{\omega}(u,v)).
\end{equation*}
Hence, every function $F$ on $\Xi$ 
satisfies the relation 
\begin{equation*}
	F\circ\Psi_v(\omega,n)=F\circ\Psi_u(\omega,n+\kappa_{\omega}(u,v)).
\end{equation*}
The topology that $\Xi$ inherits as  product of $\Omega$ and $\mathbb{Z}$ is proved to be independent of the choice of $v\in X$.

\subsection{Group actions}\label{GA}
\par
Let $G$ be the group of isometries on $X$, that is the group of bijections $g\colon X\to X$ which preserve the distance $d$. The group $G$ is unimodular and locally compact, and acts transitively on $X$ by the action
\[(g,x)\longmapsto g[x]:=g(x),\quad g\in G.\]
We fix an arbitrary reference point $o\in X$ and we denote by $K_o$ the corresponding stability subgroup. It turns out that  $K_o$ is a maximal compact subgroup of $G$ and  under the canonical bijection $g{K_o}\mapsto g[o]$ we have the identification $X\simeq G/{K_o}$ . 

The group $G$ acts on the boundary as well. Indeed, it is easy to see that if $g\in G$ and $(x_i)_{i\in\mathbb{N}}\sim(y_i)_{i\in\mathbb{N}}$, then $(g[x_i])_{i\in\mathbb{N}}\sim (g[y_i])_{i\in\mathbb{N}}$, so that  the transitive action of $G$ on $X$  induces a transitive action of $G$ on $\Omega$. Indeed, if $(x_i)_{i\in\mathbb{N}}\in c(X)$, then $(g[x_i])_{i\in\mathbb{N}}\in c(X)$ as well, because 
\[d(g[x_i],g[x_{i+1}])=d(x_i,x_{i+1})=1\]
and $g[x_i]\neq g[x_{i+2}]$ since $g\in G$. Furthermore, $(x_i)_{i\in\mathbb{N}}\sim(y_i)_{i\in\mathbb{N}}$ implies that there exist $m\in\mathbb{Z}$ and $N\in\mathbb{N}$ such that $d(g[x_i],g[y_{i+m}])=d(x_i,y_{i+m})=0$ for every $i\geq N$ and then $(g[x_i])_{i\in\mathbb{N}}\sim (g[y_i])_{i\in\mathbb{N}}$.
Precisely, the group  $G$ acts on $\Omega$ by the action 
\[
(g,\omega)\longmapsto g\cdot\omega:=p((g[x_i])_{i\in\mathbb{N}}),\qquad\omega=p((x_i)_{i\in\mathbb{N}}).
\]
This, in turn, induces a transitive action of $K_o$ on the set $\Gamma_o$ of infinite chains  starting at~$o$ by means of
\[
(k,[o,\omega))\longmapsto [o,k\cdot\omega),\quad \omega\in\Omega.
\]
We fix $\omega_0\in\Omega$ and we denote by $K_{o,\omega_0}$ the stabilizer of $[o,\omega_0)$ in $K_o$, so that   $\Gamma_o\simeq K_o/K_{o,\omega_0}$.

The group $G$ of isometries of $X$ acts transitively also on the space  $\Xi$ of horocycles through the action on vertices because the $G$-action maps horocycles in themselves. Indeed, if  $\xi\in\Xi$, $\xi=h_{\omega,n}^v$, with $v\in X$, $\omega\in\Omega$, $n\in\mathbb{N}$ and $[v,\omega)=(x_i)_{i\in\mathbb{N}}$, then for every $g\in G$
\begin{align}\label{eq:actionhorocycles}
	\nonumber g[\xi]=\{g[x]:x\in X,\, \kappa_{\omega}(v,x)=n\}&=\{g[x]: x\in X,\, \lim_{i\to\infty}(i-d(x,x_i))=n\}\\
	\nonumber &=\{x\in X: \lim_{i\to\infty}(i-d(g^{-1}[x],x_i))=n\}\\
	\nonumber &=\{x\in X: \lim_{i\to\infty}(i-d(x,g[x_i]))=n\}\\
	\nonumber &=\{x\in X: \kappa_{g{\cdot}\omega}(g[v],x)=n\}\\
	&=h_{g{\cdot}\omega,n}^{g[v]},
\end{align}
by Proposition~\ref{prop:limitk}.
Therefore $G$ acts transitively on $\Xi$ by 
\[
(g,h_{\omega,n}^{v})\longmapsto g.h_{\omega,n}^{v}:=h_{g{\cdot}\omega,n}^{g[v]}.
\]
Consider the horocycle \[\xi_0=h_{\omega_0,0}^o=\{x\in X: \kappa_{\omega_0}(o,x)=0\}.\] If $[o,\omega_0)=(x_i)_{i\in\mathbb{N}}$, then 
\begin{align*}
	g.\xi_0
	=\{x\in X: \lim_{i\to\infty}(i-d(x,g[x_i]))=0\}.
\end{align*}
Hence, the isotropy subgroup at $\xi_0$ is $H=\bigcup_{j=0}^{\infty}H_j$, where $H_j$ is  the subgroup of isometries fixing the sub-path  $[x_j,\omega_0)\in c(X)$, see also \cite{bfp}. Therefore, $\Xi\simeq G/H$. 
Observe that $H$ is the isotropy subgroup  of $G$ at $h_{\omega_0,n}^o$ for every $n\in\mathbb{Z}$. Thus, by \eqref{eq:hprop},  $H$ is the isotropy subgroup  of $G$ at every horocycle tangent to $\omega_0$, namely at $h_{\omega_0,n}^v$ for every $n\in\mathbb{Z}$ and $v\in X$.
\par
Let $\tau\in G$ be a one-step translation along $(\omega_1,\omega_0)$, with $\omega_1\in \Omega\setminus\{\omega_0\}$, where $\omega_0$ is as in the definition of $H$ (see \cite{ftn} for further details on the one-step translations in $G$). Assume that if $v\in (\omega_0,\omega_1)$, then $\tau(v)\in [v,\omega_0)$. Furthermore, denote by $A$ the subgroup of $G$ generated by the powers of $\tau$.       It is easy to see that the group $A$ acts on $H$ by conjugation. Indeed, for every $m\in\mathbb{Z}$ and $g\in H$, we have
\begin{align*}
	\tau^{m}g\tau^{-m}.\,\xi_0&=\tau^{m}g.\,h_{\omega_0,0}^{\tau^{-m}[o]}=\tau^{m}.\,h_{\omega_0,0}^{\tau^{-m}[o]}=h_{\omega_0,0}^{\tau^{m}\tau^{-m}[o]}=\xi_0,
\end{align*} 
where we use that $\tau^{m}\cdot\omega_0=\tau^{-m}\cdot\omega_0=\omega_0$.
It has been proved in \cite{veca} that the resulting semidirect product $H\rtimes A$ has modular function 
\[      \Delta(h,\tau^m)=q^m.\]
With slight abuse of notation, we write $\qdot$ for the function $\qdot\colon\mathbb{Z}\rightarrow\mathbb{{\mathbb R}_+}$ defined by
\[      \qdot(n):=q^\frac{n}{2},\]
and in what follows, the same notation is used for its trivial extension to $\Omega\times\mathbb{Z}$. The function $\Delta^\frac{1}{2}$ is the analogue of the function $e^\rho$ in the theory of symmetric spaces (see~\cite{gass}).
\\
\subsection{Measures}\label{M}
We endow $X$ with the counting measure $\de x$ which is trivially $G$-invariant, and we denote by $L^2(X)$ the Hilbert space of square-integrable functions with respect to  $\de x$.

As far as $\Omega$ is concerned, recall that $\Omega$ is identified with $\Gamma_o$ on which $K_o$ acts transitively. Therefore, $\Gamma_o$ admits a unique $K_o$-invariant probability measure $\mu^o$. We denote by $\nu^o$ the measure on $\Omega$ obtained as the  push-forward of $\mu^o$ by means of the canonical projection $p_{|_{\Gamma_o}}:\Gamma_o\to\Omega$.
It has been shown in \cite{ftn} that 
\[
\nu^o(\Omega(u))=\frac{q}{(q+1)q^{d(o,u)}},\qquad u\ne o.
\]
The measure $\nu^o$ is $G$-quasi-invariant    and, by definition, the Poisson kernel $p_o(g,\omega)$ is the associated Radon-Nikodym derivative $d\nu^o({g^{-1}}\cdot\omega)/d\nu^o(\omega)$, i.e.
\begin{equation}\label{eq:gvarinceomega}
	\int_{\Omega}F({g^{-1}}\cdot\omega)\de\nu^o(\omega)=\int_{\Omega}F(\omega)p_o(g^{-1},\omega)\de\nu^o(\omega),\qquad F\in L^1(\Omega,\nu^o),\, g\in G.
\end{equation}
It is possible to prove \cite{ftn} that 
\[
p_o(g,\omega)=q^{\kappa_\omega(o,g[o])}.
\]
Since $\nu^o$ is $K_o$-invariant, we may write $p_o(g{K_o},\omega)$ instead of $p_o(g,\omega)$. 
For every other choice of the reference vertex $v\in X$ the analogous objects $K_v, \Gamma_v, \mu^v, \nu^v, p_v$ can be introduced. It turns out that the measure $\nu^o$ is absolutely continuous with respect to $\nu^v$. Precisely 
\begin{equation}\label{verinv}
	\int_{\Omega}F(\omega)\de\nu^o(\omega)=\int_{\Omega}F(\omega)q^{\kappa_\omega(v,o)}\de\nu^v(\omega),
\end{equation}
for every $F\in L^1(\Omega,\nu^o)$.  Therefore, we can endow the boundary $\Omega$ with infinitely many measures which are absolutely continuous with respect to each other.

In order to adequately describe the measure on $\Xi$ relative to which we form the Lebesgue spaces $L^1(\Xi)$ and $L^2(\Xi)$, we need the parametrization \eqref{Psi}.
The idea is to define compatible measures on $\Omega\times\mathbb{Z}$ and $\Xi$ in the sense that the natural pull-back of functions induced by the mapping $\Psi_v\colon\Omega\times\mathbb{Z}\rightarrow\Xi$ induces  a unitary operator $\Psi_v^*$ of the corresponding $L^2$ spaces. To this end,
we consider the measure on $\mathbb{Z}$ with density $q^n$ with respect to the counting measure $\de n$. We fix $v\in X$ and  endow $\Xi$ with the measure $\lambda$ obtained as the push-forward  of the measure $\nu^v\otimes q^n\de n$ on $\Omega\times\mathbb{Z}$ by means of the map $\Psi_v$, i.e.
\[
\lambda={\Psi_v}_*(\nu^v\otimes q^n\de n),
\]
which is independent of the choice of the vertex $v$ (see \cite{bfp}). We denote by $L^1(\Xi)$ and $L^2(\Xi)$ the spaces of absolutely integrable functions and square-integrable functions with respect to $\lambda$, respectively. Thus, by definition of $\lambda$, for every $F\in L^1(\Xi)$
\begin{align*}
	\int_{\Xi} F(\xi)\de\lambda(\xi)&
	=\int_{\Omega\times\mathbb{Z}}(F\circ\Psi_v)(\omega,n)q^n\de\nu^v(\omega)\de n.
\end{align*}
It is easy to verify that $\lambda$ is $G$-invariant.   

\par    
For every $v\in X$, let $L_v^2\left(\Omega\times\mathbb{Z}\right)$ be the space of square-integrable functions w.r.t. the measure $\nu^v\otimes{\rm d}n$. For every $F\in L^2(\Xi)$, we denote by $\Psi_v^*F$ the $(L^2(\Xi),L_v^2\left(\Omega\times\mathbb{Z}\right))$-pull-back
of $F$ by $\Psi_v$, which involves the function $\Delta^\frac{1}{2}$ introduced in the previous subsection, namely
\[
\Psi_v^*F(\omega,n)=(\Delta^\frac{1}{2}\cdot(F\circ\Psi_{v}))(\omega,n),
\]
for almost every $(\omega,n)\in \Omega\times\mathbb{Z}$. Clearly, $\Psi_v^*$ is a unitary operator from $L^2(\Xi) $ into $L_v^2\left(\Omega\times\mathbb{Z}\right)$. Indeed, for every $F\in L^2(\Xi)$ we have that
\begin{align*}
	\int_{\Omega\times\mathbb{Z}}\left|\Psi_v^*F(\omega,n)\right|^2\de \nu^{v}(\omega)\de n
	=&\int_{\Omega\times\mathbb{Z}}\left|(\Delta^\frac{1}{2}\cdot(F\circ\Psi_{v}))(\omega,n)\right|^2\de \nu^v(\omega)\de n\\
	=&\int_{\Omega\times\mathbb{Z}}\left|(F\circ\Psi_{v})(\omega,n)\right|^2q^n\de \nu^v(\omega)\de n\\
	=&\int_{\Xi}\left|F(\xi)\right|^2\de\lambda(\xi)=\|F\|_{ L^2(\Xi)}^2
\end{align*}
and then $\Psi_v^*$ is an isometry from $L^2(\Xi) $ into $L_v^2\left(\Omega\times\mathbb{Z}\right)$. Surjectivity is also clear.

\subsection{Representations}\label{R}
Recall that $X$ is endowed with the counting measure $\de x$ which is trivially $G$-invariant. Thus, the group $G$  acts on $L^2(X)$ by the quasi regular representation $\pi\colon G\longrightarrow \mathcal{U}(\Le^2(X))$ defined by
\[\pi(g)f(x):=f(g^{-1}[x]),\qquad f\in\Le^2(X),\:g\in G,\]
where $\mathcal{U}(\Le^2(X))$ denotes the group of unitary operators of $\Le^2(X)$. In Appendix \ref{appendiceirri} it is shown that $\pi$ is not irreducible.  

Similarly, since $\lambda$ is $G$-invariant, the group $G$  acts on $L^2(\Xi)$ by the quasi regular unitary representation $\hat{\pi}\colon G\longrightarrow \mathcal{U}(\Le^2(\Xi))$ defined by
\[
\hat{\pi}(g)F(\xi):=F(g^{-1}.\xi),\qquad F\in\Le^2(\Xi),\:g\in G.
\]
These are the two representations in which we are interested.

	\subsection{The Helgason-Fourier transform on homogeneous trees}\label{FT}

		The Helgason-Fourier transform can be defined on homogeneous trees (see \cite{cms}, \cite{ftn}, \cite{ftp}) in analogy with the setup of symmetric spaces \cite{gass}. We briefly recall its definition and its main features. We put $T=2\pi/\log(q)$,
		$\mathbb{T}=\mathbb{R}/T\mathbb{Z}\simeq[0,T)$ and we denote by $\de t$ the normalized Lebesgue measure on $\mathbb{T}$. Let $C_c(X)$ be the space of compactly supported functions on $X$.
		\begin{defn}
			The \textit{Helgason-Fourier transform} of $f\in C_c(X)$ with respect to the vertex $v\in X$ is the function $\mathcal{H}_vf:\Omega\times\mathbb{T}\longrightarrow\mathbb{C}$ defined by
		\[	\mathcal{H}_vf(\omega,t)=\sum_{x\in X}f(x)q^{(\frac{1}{2}+it)\kappa_\omega(v,x)},\qquad(\omega,t)\in\Omega\times\mathbb{T}.\]
		\end{defn}

		As the Euclidean Fourier transform, the Helgason-Fourier transform extends to a unitary operator on $\Le^2(X)$ (see \cite{ftn}, \cite{ftp}). The Plancherel measure involves a version of the Harish-Chandra $\textbf{c}$-function inspired by the symmetric space construction \cite{hc58}, namely the meromorphic function
	\[	\textbf{c}(z)=\frac{1}{q^\frac{1}{2}+q^{-\frac{1}{2}}}\frac{q^{1-z}-q^{z-1}}{q^{\frac{1}{2}-z}-q^{z-\frac{1}{2}}},\qquad z\in\mathbb{C}\textrm{ with }q^{2z-1}\neq 1.\]
	
		We put
	\begin{equation}\label{cq}
		c_q=\frac{q}{2(q+1)}
	\end{equation}
	and we denote by $L_{v,\textbf{c}}^2(\Omega\times\mathbb{T})$ the space of square-integrable functions on $\Omega\times\mathbb{T}$ w.r.t. the measure $c_q\left|\textbf{c}(1/2+it)\right|^{-2}\de \nu^v\de t$.

	\textbf{Property ${\mathbf{\sharp}}$.} We say that $f\in L_{v,\textbf{c}}^2(\Omega\times\mathbb{T})$ satisfies Property $\sharp$ if the symmetry condition
			\begin{equation}\label{invhf}
		\int_{\Omega}p_v(x,\omega)^{\frac{1}{2}-it}F(\omega,t)\de\nu^v(\omega)=\int_{\Omega}p_v(x,\omega)^{\frac{1}{2}+it}F(\omega,-t)\de\nu^v(\omega),
	\end{equation}
	holds for every $x\in X$ and for almost every $t\in\mathbb{T}$. We denote by $L_{v,\textbf{c}}^2(\Omega\times\mathbb{T})^\sharp$ the space of functions in $L_{v,\textbf{c}}^2(\Omega\times\mathbb{T})$ satisfying Property $\sharp$.

%
		\begin{thm}[\cite{ftn}]\label{extH}
		The Helgason-Fourier transform  $\mathcal{H}_v$ extends to a unitary isomorphism $\mathscr{H}_v$ from $L^2(X)$ onto $L_{v,\emph{\textbf{c}}}^2\left(\Omega\times\mathbb{T}\right)^\sharp$.
		\end{thm}

	\section{The horocyclic Radon transform}	
	In this section we recall the definition of the horocyclic Radon transform on homogeneous trees and its fundamental properties. 
	As already mentioned, the horocyclic Radon transform is precisely the Radon transform {\it \`a la} Helgason relative to the dual pair $(X,\Xi)$.
	The case of homogeneous trees is not covered by the general setup considered in \cite{abdd} by the authors since the quasi regular representation $\pi$ of $G$ on $L^2(X)$ is not irreducible. For this reason, we can not apply the 
	results presented in \cite{abdd} in order to obtain a unitarization theorem and we therefore adopt an approach which mimics the one used in \cite{radon} and 
	\cite{acha} in the case of the polar and the affine Radon transforms, respectively. 
	\begin{defn}
		The horocyclic Radon transform $\mathcal{R}f$ of a function $f\in C_c(X)$ is the map $\mathcal{R}f:\Xi\to\mathbb{C}$ defined by 
		\[
		\mathcal{R}f(\xi)=\sum_{x\in\xi} f(x).
		\]
	\end{defn}
	We recall that for every $v\in X$ there exists a bijection $\Psi_v\colon\Omega\times\mathbb{Z}\to \Xi$ given by
	$(\omega,n)\mapsto h_{\omega,n}^v$ and we shall write $\mathcal{R}_vf=\mathcal{R}f\circ\Psi_v$.

\begin{defn}
	Let $v\in X$. The Abel transform $\mathcal{A}_vf$ of a function $f\in C_c(X)$ is the map $\mathcal{A}_vf:\Omega\times\mathbb{Z}\to\mathbb{C}$ defined by 
		\[
		\mathcal{A}_vf(\omega,n)=\Psi_v^*(\mathcal{R}f)(\omega,n)=(\Delta^\frac{1}{2}\cdot\mathcal{R}_vf)(\omega,n).
		\]
	\end{defn}
 
 We need to introduce the Fourier transform on $L^2(\mathbb{Z})$. We denote by $L^2_T$ the space of $T$-periodic functions $f$ on $\mathbb{R}$ such that 
 \[
\|f\|_{L^2_T}^2=\int_0^T|f(t)|^2 {\rm d}t<+\infty. 
 \] 
Let $s\in L^2(\mathbb{Z})$, the Fourier transform $\mathcal{F}s$ of $s$ is defined as the Fourier series of the $T$-periodic function with Fourier coefficients $(s(n))_{n\in\mathbb{Z}}$. Precisely, 	
\[	\mathcal{F}s=\sum_{n\in\mathbb{Z}}s(n)q^{in\cdot},\]
where the series converges in $L^2_T$. The Parseval identity reads
	\begin{equation*}
	\|\mathcal{F}s\|_{L^2_T}^2=\sum_{n\in\mathbb{Z}}|s(n)|^2.
	\end{equation*}
Furthermore, if $s\in L^1(\mathbb{Z})$, for almost every $t\in\mathbb{T}$
\[	\mathcal{F}s(t)=\sum_{n\in\mathbb{Z}}s(n)q^{int}.\]

We are now ready to state the result which relates the Helgason-Fourier transform with the horocyclic Radon transform. For the reader's convenience, we include the proof. 
	\begin{prop}[Fourier Slice Theorem, version I, \cite{betori1986radon,cms}]\label{fst}
		Let $v\in X$. For every $f\in C_c(X)$ and $\omega\in\Omega$, $\mathcal{A}_vf(\omega,\cdot)\in L^1(\mathbb{Z})$ and 
\begin{equation}\label{eq:fst}
	(I\otimes\mathcal{F})\mathcal{A}_vf(\omega,t)= \mathcal{H}_vf(\omega,t),
\end{equation}
for almost every $t\in\mathbb{T}$.
	\end{prop}
	\begin{proof}
	Let $f\in C_c(X)$ and $\omega\in\Omega$. By formula \eqref{eq:disjointunion} 
	\[
	\sum_{n\in\mathbb{Z}}|\mathcal{A}_vf(\omega,n)|=\sum_{n\in\mathbb{Z}}q^\frac{n}{2}|\mathcal{R}_vf(\omega,n)|\leq\sum_{n\in\mathbb{Z}}q^\frac{n}{2}\sum_{x\in h_{\omega,n}^v}|f(x)|=\sum_{x\in \mathrm{supp}f}|f(x)|q^\frac{\kappa_\omega(v,x)}{2}<+\infty.
	\]
	Then, $\mathcal{A}_vf(\omega,\cdot)$ is in $L^1(\mathbb{Z})$ and applying again \eqref{eq:disjointunion} we have that for almost every $t\in\mathbb{T}$
		\begin{align*}
		(I\otimes\mathcal{F})\mathcal{A}_vf(\omega,t)&=\sum_{n\in \mathbb{Z}}\mathcal{A}_vf(\omega,n)q^{itn}=\sum_{n\in \mathbb{Z}}q^\frac{n}{2}\mathcal{R}_vf(\omega,n)q^{itn}\\
		&=\sum_{n\in \mathbb{Z}}q^\frac{n}{2}\!\sum_{x\in h_{\omega,n}^v}\!f(x)q^{itn}
		=\sum_{x\in X}f(x)q^{(\frac{1}{2}+it)\kappa_\omega(v,x)}=\mathcal{H}_vf(\omega,t),
		\end{align*}
and this concludes the proof.
	\end{proof}
We refer to Proposition~\ref{fst} as the Fourier Slice Theorem for the horocyclic Radon transform in analogy with the polar Radon transform, see \cite{radon} as a classical reference.

In order to prove our first intertwining result, we show that $\mathcal{R}f\in L^2(\Xi)$ for every  $f\in C_c(X)$.
Let $f\in C_c(X)$ and $v\in X$. By Parseval identity and Proposition~\ref{fst} we have that 
\begin{align*}
\int_{\Xi}|\mathcal{R}f(\xi)|^2\de\lambda(\xi)&=\int_{\Omega\times\mathbb{Z}}|\Psi_v^*(\mathcal{R}f)(\omega,n)|^2\de\nu^v(\omega)\de n\\
&=\int_{\Omega\times\mathbb{T}}|(I\otimes\mathcal{F})(\Psi_v^*(\mathcal{R}f))(\omega,t)|^2\de\nu^v(\omega)\de t\\
&=\int_{\Omega\times\mathbb{T}}|\mathcal{H}_vf(\omega,t)|^2\de\nu^v(\omega)\de t.
\end{align*}
Since $f$ has finite support, then by the definition of the Helgason-Fourier transform, the inequality $|\kappa_{\omega}(v,x)|\leq d(v,x)$ and $\nu^v(\Omega)=1$, the above leads to 
 \begin{align*}
\int_{\Xi}|\mathcal{R}f(\xi)|^2\de\lambda(\xi)&=\int_{\Omega\times\mathbb{T}}|\sum_{x\in\text{supp}f}f(x)q^{(\frac{1}{2}+it)\kappa_\omega(v,x)}|^2\de\nu^v(\omega)\de t\\
&\leq\int_{\Omega}(\sum_{x\in\text{supp}f}|f(x)|q^{\frac{\kappa_\omega(v,x)}{2}})^2\de\nu^v(\omega)\\
&\leq\int_{\Omega}\sum_{x\in\text{supp}f}|f(x)|^2\sum_{x\in\text{supp}f}q^{\kappa_\omega(v,x)}\de\nu^v(\omega)\\
&=\sum_{x\in\text{supp}f}|f(x)|^2\sum_{x\in\text{supp}f}\int_{\Omega}q^{\kappa_\omega(v,x)}\de\nu^v(\omega)\\
&\leq\sum_{x\in\text{supp}f}|f(x)|^2\sum_{x\in\text{supp}f}\int_{\Omega}q^{d(v,x)}\de\nu^v(\omega)\\
&=\sum_{x\in\text{supp}f}|f(x)|^2\sum_{x\in\text{supp}f}q^{d(v,x)}<+\infty.
\end{align*}
Therefore, $\mathcal{R}f\in L^2(\Xi)$ for every $f\in C_c(X)$.
	The horocyclic Radon transform intertwines the regular representations of $G$ on $L^2(X)$ and $L^2(\Xi)$. This result is a direct
	consequence of the fact that $X$ and $\Xi$ carry $G$-invariant measures ${\rm d} x$ and ${\rm d}\lambda$.
	
	\begin{prop}\label{interradon}
	For every $g\in G$ and $f\in C_c(X)$
	\[\mathcal{R}(\pi(g)f)=\hat{\pi}(g)(\mathcal{R}f).\]
\end{prop}
\begin{proof}
For all $g\in G$ and $f\in C_c(X)$ 
	\begin{equation*}\label{intercc}
		\mathcal{R}(\pi(g)f)(\xi)=\sum_{x\in\xi}f(g^{-1}[x])=\sum_{y\in g^{-1}.\xi}f(y)=\hat{\pi}(g)(\mathcal{R}f)(\xi),
	\end{equation*}
	for every $\xi\in\Xi$.
\end{proof}

We now introduce a closed subspace of $L^2(\Xi)$ which will play a crucial role because it is the range of the unitarization 
of the horocyclic Radon transform. 

Let $v\in X$. For every $F\in L^2(\Xi)$
\begin{align*}\label{eq:firstl2zfunction}
\|F\|^2_{L^2(\Xi)}&=\int_{\Omega}\sum_{n\in\mathbb{Z}} |\Psi_v^*F(\omega,n)|^2{\rm d}\nu^v(\omega)<+\infty.
\end{align*}
Hence, the function $\Psi_v^*F(\omega,\cdot)$ is in $L^2(\mathbb{Z})$ for almost every $\omega\in\Omega$. Moreover, by Parseval identity and Fubini theorem 
\begin{align*}
\|F\|^2_{L^2(\Xi)}&=\int_{\Omega\times\mathbb{Z}}|\Psi_v^*F(\omega,n)|^2{\rm d}\nu^v(\omega)\de n\\
&=\int_{\mathbb{T}}\int_{\Omega} |(I\otimes\mathcal{F})\Psi_v^*F(\omega,t)|^2{\rm d}\nu^v(\omega)\de t<+\infty.
\end{align*}
Then, for almost every $t\in\mathbb{T}$ the function $(I\otimes\mathcal{F})\Psi_v^*F(\cdot,t)$ is in $L^2(\Omega,\nu^v)
$
and
\[
|\int_{\Omega}(I\otimes\mathcal{F})\Psi_v^*F(\omega,t)\de\nu^v(\omega)|\leq\int_{\Omega}|(I\otimes\mathcal{F})\Psi_v^*F(\omega,t)|\de\nu^v(\omega)<+\infty.
\]

\textbf{Property ${\mathbf{\flat}}$.} We say that $F\in L^2(\Xi)$ satisfies Property $\flat$ if the symmetry condition 
\begin{equation}\label{eq:equivalentfrequencycondition}
	\int_{\Omega}(I\otimes\mathcal{F})\Psi_v^*F(\omega,t)\de\nu^v(\omega)=\int_{\Omega}(I\otimes\mathcal{F})\Psi_v^*F(\omega,-t)\de\nu^v(\omega)
\end{equation}
holds for every $v\in X$ and for almost every $t\in\mathbb{T}$. We denote by $L^2_{\flat}(\Xi)$ the space of all such functions. 

Our main results in Section~\ref{sec:three} are based on the following characterization of $L_\flat^2(\Xi)$. For every $v\in X$, we denote 
by $L_v^2(\Omega\times\mathbb{T})$ the space of square-integrable functions on $\Omega\times\mathbb{T}$ w.r.t. the measure $\nu^v\otimes{\rm d} t$.

\begin{prop}\label{prop:fundamentaloperator}
	Let $v\in X$. The operator $\Phi_v$ defined on $F\in L^2(\Xi)$ by 
	\begin{equation*}
	\Phi_vF(\omega,t)=(I\otimes\mathcal{F})\Psi_v^*F(\omega,t)=(I\otimes\mathcal{F})(\Delta^\frac{1}{2}\cdot(F\circ\Psi_{v}))(\omega,t),\qquad \text{a.e.}\, (\omega,t)\in \Omega\times\mathbb{T},
	\end{equation*}
	is an isometry from $L^2(\Xi)$ into $L_v^2\left(\Omega\times\mathbb{T}\right)$. Furthermore, 
	 for every other $u\in X$ 
	\begin{equation}\label{eq:relationthetas}
	\Phi_u F(\omega,t)=p_u(v,\omega)^{\frac{1}{2}+it}\Phi_v F(\omega,t),
	\end{equation}
	for almost every $(\omega,t)\in \Omega\times\mathbb{T}$.
Finally, a function $F$ belongs to $L_\flat^2(\Xi)$
	 if and only if $\Phi_v F$ satisfies Property $\sharp$.
	\end{prop}
By Proposition~\ref{prop:fundamentaloperator}, $F\in L_\flat^2(\Xi)$ implies that $\Phi_vF$ satisfies~\eqref{invhf} for every $v\in X$. Conversely, if 
we want to prove that a function $F\in L^2(\Xi)$ satisfies \eqref{eq:equivalentfrequencycondition} it is enough to verify that \eqref{invhf} holds true for at least one, hence every, $v\in X$. This last remark 
will prove very useful in our proofs.
\begin{proof} By Parseval identity, for every $F\in L^2(\Xi)$ we have that 
	\begin{align*}
	\int_{\Omega\times\mathbb{T}}\left|\Phi_vF(\omega,t)\right|^2\de \nu^{v}(\omega)\de t=&\int_\Omega\int_\mathbb{T}\left|(I\otimes\mathcal{F})\Psi_v^*F(\omega,t)\right|^2\de t\de \nu^{v}(\omega)\\
	=&\int_{\Omega\times\mathbb{Z}}\left|\Psi_v^*F(\omega,n)\right|^2\de \nu^v(\omega)\de n=\|F\|_{ L^2(\Xi)}^2,
	\end{align*}
	so that $\Phi_v$ is an isometry from $L^2(\Xi)$ into $L_v^2\left(\Omega\times\mathbb{T}\right)$. Now, let $u\in X$ and $F\in L^2(\Xi)$. 
	For almost every $\omega\in\Omega$ we have that
	\begin{align*}
	0&=\lim_{N\to+\infty}\int_0^T|\sum_{n=-N}^NF\circ\Psi_{u}(\omega,n)q^{\frac{n}{2}}q^{int}-\Phi_uF(\omega,t)|^2\de t\\
	&=\lim_{N\to+\infty}\int_0^T|\sum_{n=-N}^NF\circ\Psi_{v}(\omega,n+\kappa_{\omega}(v,u))q^{\frac{n}{2}}q^{int}-\Phi_uF(\omega,t)|^2\de t\\
	&=\lim_{N\to+\infty}\int_0^T|\sum_{m=-N+\kappa_{\omega}(v,u)}^{N+\kappa_{\omega}(v,u)}F\circ\Psi_{v}(\omega,m)q^{\frac{1}{2}(m-\kappa_{\omega}(v,u))}q^{it(m-\kappa_{\omega}(v,u))}-\Phi_uF(\omega,t)|^2\de t\\
	&=\lim_{N\to+\infty}\int_0^T|q^{(\frac{1}{2}+it)\kappa_{\omega}(u,v)}\sum_{m=-N+\kappa_{\omega}(v,u)}^{N+\kappa_{\omega}(v,u)}F\circ\Psi_{v}(\omega,m)q^{\frac{m}{2}}q^{imt}-\Phi_uF(\omega,t)|^2\de t
	\end{align*}
	and, since 
	\[
	\Phi_vF(\omega,t)=\lim_{N\to+\infty} \sum_{m=-N+\kappa_{\omega}(v,u)}^{N+\kappa_{\omega}(v,u)}F\circ\Psi_{v}(\omega,m)q^{\frac{m}{2}}q^{imt}
	\]
	in $L_T^2$, we conclude that relation~\eqref{eq:relationthetas} holds true.
	Finally, let  $F\in L^2(\Xi)$. For every $x\in X$ and for almost every $t\in\mathbb{T}$, \eqref{eq:relationthetas} yields
	\begin{align*}
	\int_{\Omega}p_{v}(x,\omega)^{\frac{1}{2}-it}\Phi_vF(\omega,t)\de\nu^{v}(\omega)&=\int_{\Omega}p_{v}(x,\omega)^{\frac{1}{2}-it}p_v(x,\omega)^{\frac{1}{2}+it}\Phi_xF(\omega,t)\de\nu^{v}(\omega)\\
	&=\int_{\Omega}\Phi_xF(\omega,t)p_v(x,\omega)\de\nu^{v}(\omega)\\
	&=\int_{\Omega}\Phi_xF(\omega,t)\de\nu^{x}(\omega).
	\end{align*}
	Then, for every $x\in X$ and almost every $t\in\mathbb{T}$ 
	\[
	\int_{\Omega}p_{v}(x,\omega)^{\frac{1}{2}-it}\Phi_vF(\omega,t)\de\nu^{v}(\omega)=\int_{\Omega}(I\otimes\mathcal{F})\Psi_x^*F(\omega,t)\de\nu^x(\omega).
	\]
This equality allows us to conclude that $F\in L_\flat^2(\Xi)$ if and only if $\Phi_vF$ satisfies~\eqref{invhf}
	and this concludes our proof.

\end{proof}	
\begin{cor}\label{cor:radonbemolle}
For every $f\in C_c(X)$, 
\begin{equation}\label{eq:phivradonhilbert}
\Phi_v(\mathcal{R}f)=\mathcal{H}_v f
\end{equation}
in $L_v^2(\Omega\times\mathbb{T})$ and $\mathcal{R}f\in L^2_\flat(\Xi)$.
\end{cor}
\begin{proof}
The proof follows immediately by Proposition~\ref{fst} and the fact that the Helgason-Fourier transform satisfies \eqref{invhf}.
\end{proof}

	Some comments are in order. Proposition~\ref{prop:fundamentaloperator} with Corollary~\ref{cor:radonbemolle} show that $\mathcal{R}(C_c(X))\subseteq L^2_\flat(\Xi)$ 
	and it highlights the link between the range of the Radon transform with the range of the Helgason-Fourier transform, which will play a crucial role in our main result.  
	The range $\mathcal{R}(C_c(X))$ has already been completely characterized in \cite{ctcc}. We 
	recall the result in \cite{ctcc} for completeness and in order to understand the relation with $L_{\flat}^2(\Xi)$. 
	\begin{thm}[Theorem 1, \cite{ctcc}]\label{thm:rangecharacterization}
		The range of the horocyclic Radon transform on the space of functions with finite support on $X$ is the space of continuous compactly supported functions on $\Xi$ satisfying the following two conditions
		\begin{enumerate}
			\item for some $v\in X$, hence for every $v\in X$,   $\sum\limits_{n\in\mathbb{Z}}F\circ\Psi_v(\omega,n)$ is independent of $\omega\in\Omega$;
			\item for every $v\in X$ and $n\in\mathbb{Z}$
			\begin{equation}\label{eq:radonconditionintime}
			\int_{\Omega}\Psi_v^*F(\omega,n)\de\nu^v(\omega)=\int_{\Omega}\Psi_v^*F(\omega,-n)\de\nu^v(\omega).
			\end{equation}
		\end{enumerate} 
	\end{thm}
It is worth observing that condition \eqref{eq:equivalentfrequencycondition} is the equivalent on the frequency side of equation \eqref{eq:radonconditionintime} for continuous compactly supported functions on $\Xi$. 
As it will be made clear in the next section, condition \eqref{eq:equivalentfrequencycondition} better suits our needs.

\section{Unitarization and Intertwining}\label{sec:three}
In order to obtain the unitarization for the horocyclic Radon transform that we are after, we need some technicalities. Figure~\ref{fig:diagram} might help the reader to keep track of all the spaces and operators involved in our construction.
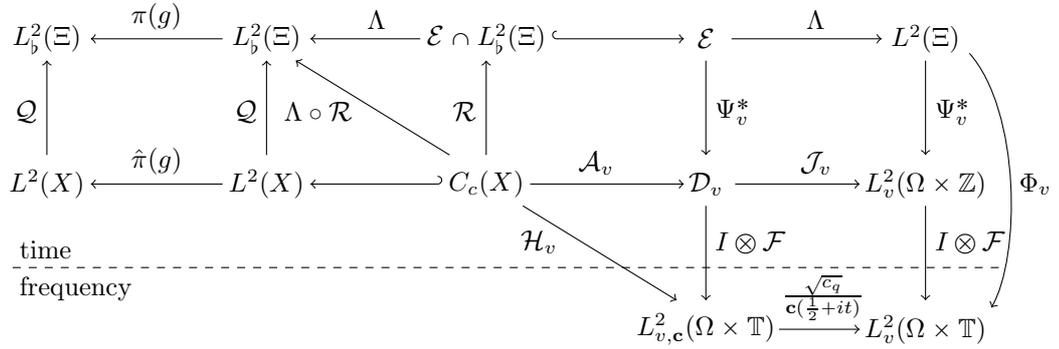
\begin{figure}[h]
	\begin{center}
		\begin{tikzpicture}[scale=0.963]
		\node at (0,0) {$L^2(\Xi)$};
		\node at (-3,0) {$\mathcal{E}$};
		\draw[<-] (-0.6,0)--(-2.65,0);
		\draw[<-] (-0.9,-2)--(-2.6,-2);
		\draw[<-] (-0.9,-4)--(-2,-4);
		\draw[->] (-3,-0.3)--(-3,-1.6);
		\node at (-3,-2) {$\mathcal{D}_v$};
		\draw[->] (0,-2.3)--(0,-3.6);
		\draw[->] (0,-0.3)--(0,-1.6);
		\draw[->] (-3,-2.3)--(-3,-3.6);
		\node at (-3,-4) {$L_{v,\mathbf{c}}^2(\Omega\times\mathbb{T})$};
		\node at (-3,-1) [anchor=west]{$\Psi_v^*$};
		\node at (-3,-2.8) [anchor=west]{$I\otimes\mathcal{F}$};
		\node at (0,-2) {$L_{v}^2(\Omega\times\mathbb{Z})$};
		\node at (1.5,-2) {$\Phi_v$};
		\node at (0,-4) {$L_{v}^2(\Omega\times\mathbb{T})$};
		\node at (-1.5,0) [anchor=south]{$\Lambda$};
		\node at (-1.5,-2) [anchor=south]{$\mathcal{J}_v$};
		\node at (-1.4,-4) [anchor=south]{$\frac{\sqrt{c_q}}{\mathbf{c}(\frac{1}{2}+it)}$};
		\node at (0,-2.8) [anchor=west]{$I\otimes\mathcal{F}$};
		\node at (0,-1) [anchor=west]{$\Psi_v^*$};
		\node at (-6,-1) [anchor=east]{$\mathcal{R}$};
		\node at (-9,-1) [anchor=east]{$\mathcal{Q}$};
		\node at (-9,-2) {$L^2(X)$};
		\draw[<-] (-9,-0.3)--(-9,-1.6);
		\node at (-6,-2) {$C_c(X)$};
		\draw[<-] (-6,-0.3)--(-6,-1.6);
		\draw[left hook->] (-6.6,-2)--(-8.4,-2);
		\draw[<-] (-8.6,-0.3)--(-6.5,-1.6);
		\node at (-7.7,-1) [anchor=east]{$\Lambda\circ\mathcal{R}$};
		\node at (-9,0) {$L_\flat^2(\Xi)$};
		\node at (-6,0) {$\mathcal{E}\cap L_\flat^2(\Xi)$};
		\draw[->] (-6.9,0)--(-8.4,0);
		\node at (-7.5,0) [anchor=south]{$\Lambda$};
		\draw[right hook->] (-5.1,0)--(-3.3,0);
		\draw[<-] (-3.4,-3.6)--(-5.5,-2.3);
		\node at (-4.9,-2.8) [anchor=east]{$\mathcal{H}_v$};
		\draw[->] (-5.4,-2)--(-3.3,-2);
		\node at (-4.5,-2) [anchor=south]{$\mathcal{A}_v$};
		\draw[->] (0.6,-0.2)..controls(1.3,-1)and(1.3,-3)..(0.9,-3.7);
		\draw[dashed] (1,-3.15)--(-12.5,-3.15);
		\node at (-12.5,-3.15) [anchor=south west] {time};
		\node at (-12.5,-3.15) [anchor=north west] {frequency};
		\node at (-12,0) {$L_\flat^2(\Xi)$};
		\node at (-12,-2) {$L^2(X)$};
		\draw[->] (-9.6,0)--(-11.4,0);
		\draw[->] (-9.6,-2)--(-11.4,-2);
		\draw[->] (-12,-1.6)--(-12,-0.3);
		\node at (-10.5,0) [anchor=south]{$\pi(g)$};
		\node at (-10.5,-2) [anchor=south]{$\hat{\pi}(g)$};
		\node at (-12,-1) [anchor=east]{$\mathcal{Q}$};
		\end{tikzpicture}
	\end{center}
	\caption{Spaces and operators that come into play in our construction.}\label{fig:diagram}
\end{figure}

 Let $v\in X$. We set
\[\mathcal{D}_v=\{\varphi\in L_v^2(\Omega\times\mathbb{Z}):(I\otimes\mathcal{F})\varphi\in L_{v,\emph{\textbf{c}}}^2(\Omega\times\mathbb{T})\}\]
and we define the operator $\mathcal{J}_v\colon\mathcal{D}_v\subseteq L_v^2(\Omega\times\mathbb{Z})\rightarrow L_v^2(\Omega\times\mathbb{Z})$ as the Fourier multiplier
\[(I\otimes\mathcal{F})(\mathcal{J}_v\varphi)(\omega,t)=\frac{\sqrt{c_q}}{\left|\textbf{c}(\frac{1}{2}+it)\right|}(I\otimes\mathcal{F})\varphi(\omega,t),\quad \text{a.e.}\, (\omega,t)\in\Omega\times\mathbb{T}, \]
where $c_q$ is given by \eqref{cq}. We define the set of functions 
\[
\mathcal{E}=\{F\in L^2(\Xi):\Phi_vF\in L_{v,\emph{\textbf{c}}}^2(\Omega\times\mathbb{T}) \}
\] 
and we consider the operator $\Lambda\colon\mathcal{E}\subseteq L^2(\Xi)\rightarrow L^2(\Xi)$ given by
\[
\Lambda F={\Psi_v^*}^{-1}\mathcal{J}_v\Psi_v^*F.
\]
\begin{lem}\label{invlam}
	The operator $\Lambda$ is independent of the choice of $v\in X$.	\end{lem}
\begin{proof}
	Take $u\in X$ and put
	\[
	\tilde{\Lambda}F={\Psi_u^*}^{-1}\mathcal{J}_u\Psi_u^*F.
	\]
	We verify  that $\Lambda=\tilde{\Lambda}$. By Proposition~\ref{prop:fundamentaloperator}, it is sufficient to prove that
	\begin{align*}
		\Phi_v(\tilde{\Lambda}F)=\Phi_v(\Lambda F)
	\end{align*}
	for every $F\in L^2(\Xi)$. For almost every $(\omega,t)\in\Omega\times\mathbb{T}$ \eqref{eq:relationthetas} yields 
	\begin{align*}
		\Phi_v(\tilde{\Lambda}F)(\omega,t)&=p_v(u,\omega)^{\frac{1}{2}+it}\Phi_u(\tilde{\Lambda}F)(\omega,t)\\
		&=p_v(u,\omega)^{\frac{1}{2}+it}(I\otimes\mathcal{F})(\mathcal{J}_u\Psi_u^*F)(\omega,t)\\
		&=p_v(u,\omega)^{\frac{1}{2}+it}\frac{\sqrt{c_q}}{\left|\textbf{c}(\frac{1}{2}+it)\right|}(I\otimes\mathcal{F})\Psi_u^*F(\omega,t)\\
		&=\frac{\sqrt{c_q}}{\left|\textbf{c}(\frac{1}{2}+it)\right|}(I\otimes\mathcal{F})\Psi_v^*F(\omega,t)\\
		&=(I\otimes\mathcal{F})(\mathcal{J}_v\Psi_v^*F)(\omega,t)=\Phi_v(\Lambda F)(\omega,t)
	\end{align*}
	and we can conclude that $\Lambda=\tilde{\Lambda}$.
\end{proof}

As a direct consequence of Lemma~\ref{invlam}, for every $v\in X$ we have that for every $F\in\mathcal{E}$ and for almost every $(\omega,t)\in \Omega\times\mathbb{T}$
\begin{align}\label{eq:philambda}
\nonumber\Phi_v(\Lambda F)(\omega,t)&=(I\otimes\mathcal{F})(\mathcal{J}_v\Psi_{v}^*F)(\omega,t)\\
\nonumber&=\frac{\sqrt{c_q}}{\left|\textbf{c}(\frac{1}{2}+it)\right|}(I\otimes\mathcal{F})(\Psi_{v}^*F)(\omega,t)\\
&=\frac{\sqrt{c_q}}{\left|\textbf{c}(\frac{1}{2}+it)\right|}\Phi_vF(\omega,t).
\end{align}
\par
The operator $\Lambda$ intertwines the regular representation $\hat{\pi}$ as shown by the next proposition.

\begin{prop}\label{interhat}
	The subspace $\mathcal{E}$ is $\hat{\pi}$-invariant and for all $F\in\mathcal{E}$ and $g\in G$ 
	\begin{equation}\label{eq:intertwininglambda}
	\hat{\pi} (g)\Lambda F=\Lambda\hat{\pi}(g)F.
	\end{equation}
\end{prop}
\begin{proof}
We consider $F\in\mathcal{E}$, $g\in G$ and we prove that $\hat{\pi}(g)F\in\mathcal{E}$. 
We observe that
\[
\hat{\pi} (g)F\circ\Psi_v(\omega,n)=F\circ\Psi_{g^{-1}[v]}(g^{-1}\cdot\omega,n)
\]
for almost every $(\omega,n)\in\Omega\times\mathbb{Z}$. Therefore, we have 
\[
\Psi_{v}^*(\hat\pi(g)F)(\omega, n)=\Psi_{g^{-1}[v]}^*F(g^{-1}\cdot\omega,n)
\]
and consequently 
\begin{equation}\label{eq:expressionpihat}
\Phi_v(\hat\pi(g)F)(\omega, t)=\Phi_{g^{-1}[v]}F(g^{-1}\cdot\omega,t)
\end{equation}
for almost every $(\omega,t)\in\Omega\times\mathbb{T}$.
By  equations~\eqref{eq:gvarinceomega}, \eqref{verinv} and \eqref{eq:expressionpihat}
\begin{align*}
&\int_{\Omega\times\mathbb{T}}|\Phi_v(\hat{\pi}(g)F)(\omega,t)|^2\frac{c_q\de\nu^{v}(\omega)\de t}{|\textbf{c}(\frac{1}{2}+it)|^2}\\
&=\int_{\mathbb{T}}\int_{\Omega}|\Phi_{g^{-1}[v]}F(g^{-1}\cdot\omega,t)|^2\frac{c_q\de\nu^{v}(\omega)\de t}{|\textbf{c}(\frac{1}{2}+it)|^2}\\
&=\int_{\mathbb{T}}\int_{\Omega}|\Phi_{g^{-1}[v]}F(\omega,t)|^2p_v(g^{-1}[v],\omega)\frac{c_q\de\nu^{v}(\omega)\de t}{|\textbf{c}(\frac{1}{2}+it)|^2}\\
&=\int_{\Omega\times\mathbb{T}}|\Phi_{g^{-1}[v]}F(\omega,t)|^2\frac{c_q\de\nu^{g^{-1}[v]}(\omega)\de t}{|\textbf{c}(\frac{1}{2}+it)|^2}<+\infty
\end{align*}
and we conclude that $\hat{\pi}(g)F\in\mathcal{E}$. We next prove the intertwining property~\eqref{eq:intertwininglambda}. We have already observed that, by Proposition~\ref{prop:fundamentaloperator}, it is enough to prove that 
\begin{align*}
\Phi_v(\hat{\pi} (g)\Lambda F)=\Phi_v(\Lambda\hat{\pi}(g)F)
\end{align*}
for every $g\in G$ and $F\in\mathcal{E}$. By  equations~\eqref{eq:philambda} and \eqref{eq:expressionpihat},
for almost every $(\omega,t)\in\Omega\times\mathbb{T}$, we have the chain of equalities 
\begin{align*}
\Phi_v(\hat{\pi} (g)\Lambda F)(\omega,t)&=\Phi_{g^{-1}[v]}(\Lambda F)(g^{-1}\cdot\omega,t)\\
&=\frac{\sqrt{c_q}}{\left|\textbf{c}(\frac{1}{2}+it)\right|}\Phi_{g^{-1}[v]}F(g^{-1}\cdot\omega,t)\\
&=\frac{\sqrt{c_q}}{\left|\textbf{c}(\frac{1}{2}+it)\right|}\Phi_{v}(\hat\pi(g)F)(\omega,t)
=\Phi_v(\Lambda\hat{\pi}(g)F)(\omega,t),
\end{align*}
which proves the intertwining relation. 
\end{proof}
The next result follows directly by Proposition~\ref{prop:fundamentaloperator} and equation~\eqref{eq:philambda}.
\begin{cor}\label{cor:lambda}
For every $F\in \mathcal{E}$, $\Lambda F\in L^2_{\flat}(\Xi)$ if and only if $F\in L^2_{\flat}(\Xi)$.
\end{cor}
\begin{proof}
By Proposition~\ref{prop:fundamentaloperator}, $\Lambda F\in L^2_{\flat}(\Xi)$ if and only if $\Phi_v(\Lambda F)$ satisfies~\eqref{invhf}. 
By \eqref{eq:philambda} and since $t\mapsto \left|\textbf{c}(1/2+it)\right|$ is even, $\Phi_v(\Lambda F)$ satisfies~\eqref{invhf} if and only if $\Phi_v(F)$ satisfies~\eqref{invhf}, which is equivalent 
to $F\in L^2_\flat(\Xi)$. This concludes the proof. 
\end{proof}

\noindent
We are now in a position to prove our main result.
\begin{thm}\label{thm:unitarizationtheorem}
	The composite operator $\Lambda\mathcal{R}$ extends to a unitary operator $$\mathcal{Q}\colon L^2(X)\longrightarrow  L_\flat^2(\Xi)$$ which intertwines the representations $\pi$ and $\hat{\pi}$, i.e.
	\begin{equation}\label{intertw}
	\hat{\pi}(g)\mathcal{Q}=\mathcal{Q}\pi(g),\hspace{8mm}g\in G.
	\end{equation}
\end{thm}
Theorem~\ref{thm:unitarizationtheorem} implies that $\hat{\pi}$ is not irreducible, too. 
\begin{proof}
	We first show that $\Lambda\mathcal{R}$ extends to a unitary operator $\mathcal{Q}$ from $L^2(X)$ onto $L_\flat^2(\Xi)$. Let $f\in C_c(X)$ and $v\in X$. By the Fourier Slice Theorem \eqref{eq:fst}, the Parseval identity and the definition of $\Lambda$, we have that
	\begin{align*}
	\|f\|_{ L^2(X)}^2&=\|\mathcal{H}_{v} f\|_{L_{v,\emph{\textbf{c}}}^2(\Omega\times\mathbb{T})^\sharp}^2\\
	&=\|(I\otimes\mathcal{F})(\Psi_{v}^*(\mathcal{R}f))\|_{L_{v,\emph{\textbf{c}}}^2(\Omega\times\mathbb{T})^\sharp}^2\\
	&=\int_{\Omega\times\mathbb{T}}|(I\otimes\mathcal{F})(\mathcal{J}_v\Psi_{v}^*(\mathcal{R}f))(\omega,t)|^2\de\nu^{v}(\omega)\de t\\
	&=\int_{\Omega\times\mathbb{T}}|(I\otimes\mathcal{F})(\Psi_{v}^*(\Lambda\mathcal{R}f))(\omega,t)|^2\de\nu^{v}(\omega)\de t\\
	&=\int_{\Omega\times\mathbb{Z}}|\Psi_{v}^*(\Lambda\mathcal{R}f)(\omega,n)|^2\de\nu^{v}(\omega)\de n\\
	&=\|\Lambda\mathcal{R}f\|_{L^2(\Xi)}^2.
	\end{align*}
	Hence, $\Lambda\mathcal{R}$ is an isometric operator from $C_c(X)$ into $L^2(\Xi)$. Since $C_c(X)$ is dense in $\Le^2(X)$, $\Lambda\mathcal{R}$ extends to a unique isometry from $\Le^2(X)$ onto the closure of $\mathrm{Ran}(\Lambda\mathcal{R})$ in $L^2(\Xi)$. We must show that $\Lambda\mathcal{R}$ has dense image in $L_\flat^2(\Xi)$. The inclusion $\textrm{Ran}(\Lambda\mathcal{R})\subseteq L_\flat^2(\Xi)$ follows immediately from Corollary~\ref{cor:radonbemolle} and Corollary~\ref{cor:lambda}.
	Let $F\in L_\flat^2(\Xi)$ be such that $\langle F,\Lambda\mathcal{R}f\rangle_{L^2(\Xi)}=0$ for every $f\in C_c(X)$. By the Parseval identity and the Fourier Slice Theorem \eqref{eq:fst} we have that
	\begin{align*}
	0&=\langle F,\Lambda\mathcal{R}f\rangle_{L^2(\Xi)}\\
	&=\int_{\Omega\times\mathbb{Z}}(F\circ\Psi_{v})(\omega, n)\overline{(\Lambda\mathcal{R}f\circ\Psi_{v})(\omega, n)}q^n\de\nu^{v}(\omega)\de n\\
	&=\int_{\Omega\times\mathbb{Z}}(\Psi_{v}^*F)(\omega, n)\overline{(\mathcal{J}_v\Psi_{v}^*(\mathcal{R}f))(\omega, n)}\de\nu^{v}(\omega)\de n\\
	&=\int_{\Omega\times\mathbb{T}}\Phi_v(F)(\omega, t)\overline{(I\otimes\mathcal{F})(\mathcal{J}_v\Psi_{v}^*(\mathcal{R}f))(\omega, t)}\de\nu^{v}(\omega)\de t\\
	&=\int_{\Omega\times\mathbb{T}}\Phi_v(F)(\omega, t)\overline{(I\otimes\mathcal{F})(\Psi_{v}^*(\mathcal{R}f))(\omega, t)}\frac{\sqrt{c_q}\de\nu^{v}(\omega)\de t}{|\textbf{c}(\frac{1}{2}+it)|}\\
	&=\int_{\Omega\times\mathbb{T}}\frac{|\textbf{c}(\frac{1}{2}+it)|}{\sqrt{c_q}}\Phi_v(F)(\omega,t)\overline{\mathcal{H}_{v}f(\omega, t)}\frac{c_q\de\nu^{v}(\omega)\de t}{|\textbf{c}(\frac{1}{2}+it)|^2}.
	\end{align*}
	For simplicity of notation, we denote by $\Theta_vF$ the function on $\Omega\times\mathbb{T}$ defined as 
	\[
	\Theta_vF(\omega,t)=\frac{|\textbf{c}(\frac{1}{2}+it)|}{\sqrt{c_q}}\Phi_v(F)(\omega,t),\qquad \text{a.e.}\, (\omega,t)\in\Omega\times\mathbb{T}.
	\]
	Hence we have proved that $\langle\Theta_v F, \mathcal{H}_v f\rangle=0$
 for every $f\in C_c(X)$. The following two facts follow immediately from Proposition~\ref{prop:fundamentaloperator}. Since $\Phi_v$ is an isometry from $L^2(\Xi)$ into $L_v^2\left(\Omega\times\mathbb{T}\right)$, then $\Theta_vF$ belongs to $L_{v,\emph{\textbf{c}}}^2(\Omega\times\mathbb{T})$. Furthermore, since $F\in L_\flat^2(\Xi)$ and since $t\mapsto \left|\textbf{c}(1/2+it)\right|$ is even, then $\Theta_vF	\in \Le_{v,\emph{\textbf{c}}}^2(\Omega\times\mathbb{T})^\sharp$.
	By Theorem \ref{extH}, $\mathcal{H}_{v}(C_c(X))$ is dense in $\Le_{v,\emph{\textbf{c}}}^2(\Omega\times\mathbb{T})^\sharp$. Thus, $\Theta_vF=0$ in $\Le_{v,\emph{\textbf{c}}}^2(\Omega\times\mathbb{T})^\sharp$ 
	and then $\Phi_v(F)=0$ in $L_v^2\left(\Omega\times\mathbb{T}\right)$. Since $\Phi_v$ is an isometry from $L^2(\Xi)$ into $L_v^2\left(\Omega\times\mathbb{T}\right)$, then $F=0$ in $L^2(\Xi)$. 
	Therefore, $\overline{\mathrm{Ran}(\Lambda\mathcal{R})}=L_\flat^2(\Xi)$ and $\Lambda\mathcal{R}$ extends uniquely to a surjective isometry
	$$\mathcal{Q}\colon \Le^2(X)\longrightarrow L_\flat^2(\Xi).$$
	Observe that $\mathcal{Q}f=\Lambda\mathcal{R}f$ for every $f\in C_c(X)$.
	Then, the intertwining property \eqref{intertw} follows immediately from Proposition~\ref{interradon} and Proposition~\ref{interhat}.
	\end{proof}

As a byproduct, one obtains an extended Fourier Slice Theorem.
	\begin{prop}[Fourier Slice Theorem, version II]
		Let $v\in X$. For every $f\in L^2(X)$
		$$(I\otimes\mathcal{F})(\Psi_{v}^*(\mathcal{Q}f))(\omega,t)=\frac{\sqrt{c_q}}{|\mathbf{c}(\frac{1}{2}+it)|}\mathscr{H}_vf(\omega,t)$$
		for almost every $(\omega,t)\in\Omega\times\mathbb{T}$.
	\end{prop}
	\begin{proof}
		Let $v\in X$. For every $f\in C_c(X)$, by \eqref{eq:phivradonhilbert} and \eqref{eq:philambda} we have that 
	\begin{align*}
		(I\otimes\mathcal{F})(\Psi_{v}^*(\mathcal{Q}f))(\omega,t)&=\Phi_v(\mathcal{Q}f)(\omega,t)\\
		&=\Phi_v(\Lambda\mathcal{R}f)(\omega,t)\\
		&=\frac{\sqrt{c_q}}{|\mathbf{c}(\frac{1}{2}+it)|}\Phi_v(\mathcal{R}f)(\omega,t)\\
		&=\frac{\sqrt{c_q}}{|\mathbf{c}(\frac{1}{2}+it)|}\mathcal{H}_vf(\omega,t),
			 \end{align*}
			 for almost every $(\omega,t)\in\Omega\times\mathbb{T}$.
			Let $f\in L^2(X)$, since $C_c(X)$ is dense in $L^2(X)$, then there exists a sequence $(f_m)_m\subseteq C_c(X)$ such that $f_m\rightarrow f$ in $L^2(X)$. Then, since $\mathcal{Q}$ is a unitary operator from $L^2(X)$ onto $L_\flat^2(\Xi)$ and $\Phi_v$ is an isometry from $L^2(\Xi)$ into $L_v^2\left(\Omega\times\mathbb{T}\right)$, then $\Phi_v(\mathcal{Q}f_m)\to\Phi_v(\mathcal{Q}f)$ in $L_{v}^2\left(\Omega\times\mathbb{T}\right)$.
			Since $f_m\in C_c(X)$ for every $m\in\mathbb{N}$, 
			\[(I\otimes\mathcal{F})(\Psi_{v}^*(\mathcal{Q}f_	m))(\omega,t)=\frac{\sqrt{c_q}}{|\mathbf{c}(\frac{1}{2}+it)|}\mathcal{H}_vf_m(\omega,t),\]
			for almost every $(\omega,t)\in\Omega\times\mathbb{T}$. Hence, passing to a subsequence if necessary, for almost every $(\omega,t)\in\Omega\times\mathbb{T}$
			\[\lim_{m\rightarrow+\infty}\frac{\sqrt{c_q}}{|\mathbf{c}(\frac{1}{2}+it)|}\mathcal{H}_vf_m(\omega,t)=(I\otimes\mathcal{F})(\Psi_{v}^*(\mathcal{Q}f))(\omega,t).\]
		Therefore, passing to a subsequence if necessary, for almost every $(\omega,t)\in\Omega\times\mathbb{T}$
			\[(I\otimes\mathcal{F})(\Psi_{v}^*(\mathcal{Q}f))(\omega,t)=\lim_{m\rightarrow+\infty}\frac{\sqrt{c_q}}{|\mathbf{c}(\frac{1}{2}+it)|}\mathcal{H}_vf_m(\omega,t)=\frac{\sqrt{c_q}}{|\mathbf{c}(\frac{1}{2}+it)|}\mathscr{H}_vf(\omega,t)\]
			and this concludes our proof.
	\end{proof}

\appendix 

\section{The Radon Transform between Dual Pairs}\label{appendiceHel}

The inversion of the Radon transform consists in reconstructing an unknown signal $f$ on $\mathbb{R}^d$ from its integrals over hyperplanes. This classical inverse problem generalizes in the question of recovering an unknown function on a manifold $X$ by means of its integrals over a family $\Xi$ of submanifolds. Motivated by the group structure shared by classical examples, Helgason  introduced a natural framework for such general inverse problems by modelling $X$ and $\Xi$ as two homogeneous spaces of the same locally compact group $G$. This Appendix is devoted to recall his general framework.

Let $G$ be a locally compact group. We consider a space $X$ on which $G$ acts transitively, and we denote the action on $x\in X$ by
  \[(g,x) \mapsto g[x]. \]
We fix $x_0\in X$ and we denote by $K$ the corresponding stability subgroup, so that $X\simeq G/K$ under the canonical  isomorphism $gK\mapsto g[x_0]$.  Now, we fix a closed subgroup $H$ of $G$ and we define the root manifold $\xi_0$ as
\[
\xi_0=H[x_0]\subset X.
\]
Then, for every $gH\in G/H$ we define 
\[
\xi=g[\xi_0]=gH[x_0]\subset X,
\]
which is independent of the choice of the representative $g$ of $gH\in G/H$. We set $\Xi=\{g[\xi_0]:g\in G\}$. By definition, $G$ acts transitively on $\Xi$ by the action
  \[(g,\xi) \mapsto g.\,\xi=gg'[\xi_0],\qquad \xi=g'[\xi_0], \]
and we denote by $\widetilde{H}$ the stability subgroup of $\xi_0$. Hence, $\Xi\simeq G/\widetilde{H}$ under the canonical  isomorphism $g\widetilde{H}\mapsto g.\xi_0$. We require that 
\begin{equation}\label{eq:conditioninjectivity}
H=\widetilde{H}.
\end{equation} 
By definition, $\xi_0$ is an $H$-transitive space, hence it admits a quasi-invariant measure. In Helgason's approach $\xi_0$ is supposed to carry an $H$-invariant measure, that is 
$$
\int_{\xi_0}f(h^{-1}[x]){\rm d} m_0(x)=\int_{\xi_0}f(x){\rm d} m_0(x),\qquad g\in L^1(\xi_0,{\rm d} m_0),\, h\in H.
$$
In order to define the Radon transform, we push-forward the measure $d m_0$ to $\xi=gH$ by the map $\xi_0\ni x\mapsto g[x]\in\xi$. We denote by ${\rm d} m_\xi$ the so obtained measure on $\xi$. Since the measure on $\xi_0$ is supposed to be $H$-invariant, then the measure ${\rm d} m_\xi$ does not depend on the choice of the representatives of $\xi$. 
\begin{defn}\label{defn:radon}
We define the Radon transform of $f\colon X\to \mathbb{C}$ as the map $\mathcal{R} f:\Xi\to\mathbb{C}$ given by
\begin{equation}\label{radondualradon1}
\mathcal{R} f(\xi)=\int_{\xi}f(x){\rm d} m_\xi(x),
\end{equation}
for any $f$ for which the integral converges. 
\end{defn}
Interchanging the roles of $X$ and $\Xi$, we can define 
	\[
	\check{x}_0=K.\, \xi_0,
	\]
	and for every $x=gK\in G/K$ we set 
	\[
	\check{x}=g.\, \check{x}_0=gK.\, \xi_0,
	\]
	which is independent of the choice of the representative $g$ of $gK\in G/K$.
	We may think of $\check{x}$ as the sheaf of  manifolds in $\Xi$ passing through $x\in X$. By definition, $G$ acts transitively on the orbit $\{g.\check{x}_0:g\in G\}$ and we denote by $\widetilde{K}$ the stability subgroup at $\check{x}_0$. The conditions 
	\begin{equation}\label{eq:transversalityconditionsecond}
		K=\widetilde{K}\qquad \text{and}\qquad H=\widetilde{H}
	\end{equation}
	are known as transversality conditions, and a pair $(X,\Xi)$ which satisfies \eqref{eq:transversalityconditionsecond} is called by Helgason a dual pair {\cite[Chapter II]{radon}}. The reader may consult {\cite[Chapter II]{radon}} for numerous examples of dual pairs.
By definition, $\check{x}_0$ is a $K$-transitive space, and it is supposed to carry a $K$-invariant measure, that is 
$$
\int_{\check{x}_0}F(k^{-1}.\xi){\rm d}\mu_0(\xi)=\int_{\check{x}_0}F(\xi){\rm d}\mu_0(\xi),\qquad F\in L^1(\check{x}_0,{\rm d}\mu_0),\, k\in K.
$$
We push-forward the measure $d \mu_0$ to $\check{x}=(gK)^{\check{}}$ by the map $\check{x}_0\ni\xi\mapsto g.\xi\in\check{x}$. Since the measure $\check{x}_0$ is $K$-invariant, the so obtained measure ${\rm d}\mu_x$ does not depend on the choice of the representative of $x$. 
\begin{defn}
The dual Radon transform of $F$ is the map  $\mathcal{R}^{\#} F\colon X\to\mathbb{C}$ given by 
\begin{equation}\label{dualradon}
\mathcal{R}^{\#} F(x)=\int_{\check{x}}F(\xi){\rm d}\mu_x(\xi),
\end{equation}
for any $F\colon\Xi\to\mathbb{C}$ for which the integral converges. 
\end{defn}
We conclude this Appendix showing that the horocyclic Radon transform on homogeneous trees recalled in Section~2 is precisely the Radon transform {\it \`a la} Helgason for the dual pair $(X,\Xi)$, where $X$ is an homogeneous tree and $\Xi$ is the set of horocycles on $X$. We keep the notation of Section~\ref{GA}. Once we have fixed the origins $o\in X$ and $\omega_0\in\Omega$, and consequently the closed subgroup $H$ of $G={\rm Aut}(X)$, we define the root horocycle $\xi_0$ as
	\[
	\xi_0=H[o].
	\]
	Then, for every $gH\in G/H$ we define 
	\[
	\xi=g[\xi_0]=gH[o].
	\]
	By direct computation, $\xi_0=h_{\omega_0,0}^o$, and \eqref{eq:actionhorocycles} implies that $\xi=h_{g\cdot\omega_0,0}^{g[o]}$. So that, $\Xi=\{g[\xi_0]:g\in G\}$ is exactly the set of horocycles in $X$.
	By the definition of $H$, we have that 
	\begin{equation*}
		H=\{g\in G: g[\xi_0]=\xi_0\}.
	\end{equation*}
	Thus, condition \eqref{eq:conditioninjectivity} is satisfied and $\Xi\simeq G/H$. We endow the root horocycle $\xi_0$ with the counting measure $d \mu_0$, which is an $H$-invariant measure. Then, all the horocycles in $\Xi$ are equipped with the counting measure by pushing forward $d \mu_0$ to $\xi=gH$ by the map $\xi_0\ni x\mapsto g[x]\in\xi$. It is therefore clear that the horocyclic Radon transform on homogeneous trees is precisely the Radon transform in Definition~\ref{defn:radon} when $X$ is an homogeneous tree and $\Xi$ is the family of horocycles on $X$.

\section{The quasi regular representation is not irreducible}\label{appendiceirri}
		The quasi-regular representation $\pi$ of $G={\rm Aut}(X)$ on $L^2(X)$ endowed with the counting measure is defined by $\pi(g)f(x)=f(g^{-1}[x])$, $f\in L^2(X)$. 
		For clarity, we include a short proof that $\pi$ is not irreducible.
		
		Our approach is based on the characterization of irreducibility given by Proposition~2.47 in \cite{librorosso}. The representation $\pi$ is not irreducible if and only if there exist two functions $h_1,\,h_2\in L^2(X)\setminus\{0\}$ such that $\langle h_1,\pi(\,\cdot\,)h_2\rangle_{L^2(X)}$ vanishes identically on $G$. We start by proving that for $f\in L^2(X)$ and $g\in G$, the action of $G$ on $X$ in frequency reads
		\[\mathscr{H}_v(\pi(g)f)(\omega,t)=q^{\left(\frac{1}{2}+it\right)\kappa_{\omega}(v,g[v])}\mathscr{H}_vf(g^{-1}\!.\,\omega,t),\quad (\omega,t)\in\Omega\times\mathbb{T}.\] 
		By the density of $C_c(X)$ in $L^2(X)$, it is sufficient to prove it for $f\in C_c(X)$. Indeed,
		\begin{align*}
			\mathcal{H}_v(\pi(g)f)(\omega,t)&=\sum_{x\in X}f(g^{-1}[x])q^{\left(\frac{1}{2}+it\right)\kappa_{\omega}(v,x)}\\
			&=\sum_{x\in X}f(x)q^{\left(\frac{1}{2}+it\right)\kappa_{\omega}(v,g[x])}\\
			&=q^{\left(\frac{1}{2}+it\right)\kappa_{\omega}(v,g[v])}\sum_{x\in X}f(x)q^{\left(\frac{1}{2}+it\right)\kappa_{\omega}(g[v],g[x])}\\
			&=q^{\left(\frac{1}{2}+it\right)\kappa_{\omega}(v,g[v])}\sum_{x\in X}f(x)q^{\left(\frac{1}{2}+it\right)\kappa_{g^{-1}\!.\,\omega}(v,x)}\\
			&=q^{\left(\frac{1}{2}+it\right)\kappa_{\omega}(v,g[v])}\mathcal{H}_vf(g^{-1}\!.\,\omega,t).
		\end{align*}
		
		Now we want to find two not zero functions of $L^2(X)$ whose corresponding coefficient vanishes identically on $G$. We introduce the subset \[A:=\Omega\times\left[-\frac{T}{4},\frac{T}{4}\right]\subseteq\Omega\times\mathbb{T}.\] 
		Take $f\in L^2(X)$. We know that $\mathscr{H}_vf\in L_{v,\mathbf{c}}^2(\Omega\times\mathbb{T})^\sharp $ and if we multiply $\mathscr{H}_vf$ by the characteristic function of $A$ or $A^c$ it remains in $L_{v,\mathbf{c}}^2(\Omega\times\mathbb{T})^\sharp $ since~\eqref{invhf} is true if the function is restricted to a symmetric subset. We therefore choose 
		\[h_1=\mathscr{H}_v^{-1}(\chi_A\mathscr{H}_vf),\quad h_2=\mathscr{H}_v^{-1}(\chi_{A^c}\mathscr{H}_vf).\]
		Observe that the coefficient associated to $h_1$ and $h_2$ is 
		\begin{align*}
			\langle &h_1,\pi(g)h_2\rangle_{L^2(X)} =\langle \chi_A\mathscr{H}_vf,\mathscr{H}_v(\pi(g)h_2)\rangle_{L_{v,\mathbf{c}}^2(\Omega\times\mathbb{T})^\sharp} \\
			&=\int_{\Omega\times\mathbb{T}} \chi_A(\omega,t)\chi_{A^c}(g^{-1}\!.\,\omega,t)\mathscr{H}_vf(\omega,t)\overline{\mathscr{H}_vf(g^{-1}\!.\,\omega,t)}q^{\left(\frac{1}{2}-it\right)\kappa_{\omega}(v,g[v])}\frac{c_q\de\nu^{v}(\omega)\de t}{|\textbf{c}(\frac{1}{2}+it)|^2}=0
		\end{align*}
		Finally by Proposition~2.47 in \cite{librorosso}, we conclude that $\pi$ is not irreducible.

\section*{Acknowledgement}

F. De Mari is member of the Gruppo Nazionale per l'Analisi Matematica, la Probabilità e le loro Applicazioni (GNAMPA) of the Istituto Nazionale di Alta Matematica (INdAM) and together with F. Bartolucci and M. Monti are part of the Machine Learning Genoa Center (MaLGa). We thank the anonymous referees for helping us in substantially improving the presentation.

\end{document}